\documentclass[]{amsart}

\makeatletter
\@addtoreset{equation}{section}

\makeatother

\makeatletter
\@removefromreset{equation}{section}
\makeatother

\newtheorem{theorem}{Theorem}
\newtheorem{proposition}{Proposition}

\newtheorem{lemma}{Lemma}
\theoremstyle{definition}
\newtheorem{definition}{Definition}

\newtheorem{remark}{Remark}

\makeatletter
\@removefromreset{theorem}{section}
\makeatother

\DeclareMathOperator{\sech}{sech}
\DeclareMathOperator{\cn}{cn}
\DeclareMathOperator{\dn}{dn}
\DeclareMathOperator{\sn}{sn}
\DeclareMathOperator{\am}{am}

\DeclareMathOperator{\supp}{supp}

\usepackage{braket} 
\usepackage{amssymb} 
\usepackage{a4wide}
\usepackage{color}
\usepackage[pdftex]{graphicx}
\usepackage{mathrsfs} 
\usepackage{url}
\usepackage{hyperref}

\begin{document}

\title[]{Obstacle problems for the elastic flow and related topics}

\author{Kensuke Yoshizawa}
\address[K.~Yoshizawa]{Faculty of Education, Nagasaki University,
1-14 Bunkyo-machi, Nagasaki, 852-8521, Japan
}
\email{k-yoshizaw@nagasaki-u.ac.jp}

\subjclass[2020]{}

\date{\today}

\begin{abstract}
  In this note, we study an obstacle problem for the elastic flow. We prove the local-in-time existence of weak solutions and discuss their relation to classical solutions when additional regularity is obtained. 
  Related results concerning obstacle problems for the bending energy are also collected.
\end{abstract}

\maketitle


\section{Introduction}\label{sect:intro}
Elastic flow describes the evolution of planar curves driven by the gradient descent of bending energy penalized by length.
For an immersed curve $\gamma$ in $\mathbf{R}^2$, the \emph{bending energy} is defined by
\[
\mathcal{B}[\gamma]:=\int_\gamma |\kappa|^2\;\mathrm{d}s, 
\]
where $s$ denotes the arclength parameter of $\gamma$, and $\kappa:=\partial_s^2\gamma$ denotes the curvature vector of $\gamma$ (with $\partial_s:=|\partial_x \gamma|^{-1}\partial_x$).
For a fixed $\lambda>0$, we define the penalized bending energy by  $\mathcal{E}_\lambda:=\mathcal{B} + \lambda \mathcal{L}$, where $\mathcal{L}[\gamma]:=\int_\gamma \mathrm{d}s$.
Let $I:=(0,1)$ and $\bar{I}:=[0,1]$. 
We call $\gamma:\bar{I}\times[0,T]\to\mathbf{R}^2$ a \emph{(length-penalized) elastic flow} if $\gamma$ satisfies
\begin{align}\label{eq:elastic_flow}
    \partial_t \gamma = -\nabla_{L^2}\mathcal{E}_\lambda[\gamma], 
\end{align} 
where $\nabla_{L^2}\mathcal{E}_\lambda[\gamma]$ denotes the $L^2(\mathrm{d}s)$-gradient of $\mathcal{E}_\lambda$. 
It is known that this gradient can be expressed explicitly as $\nabla_{L^2}\mathcal{E}_\lambda[\gamma]=\textcolor{black}{2\nabla_s^2\kappa+|\kappa|^2\kappa-\lambda \kappa}$, where $\nabla_s f := \partial_s f -\langle \partial_sf, \partial_s \gamma \rangle\partial_s \gamma$ denotes the normal derivative along $\gamma$. 
The length-penalized elastic flow has been extensively studied; for related literature, we refer to \cite{DPS16, Dia24, DKS02, KM24, MPP21, MP21, Miura25, MY_JDE, NO_2014, NO_2017, PolPhD, Spe17}.
Nowadays there are also various studies related to the length-penalized elastic flow; e.g.\ the length-preserving elastic flow, the elastic flow of networks, the $p$-elastic flow, and so on. 

In this note we consider a version of an obstacle problem for elastic flow. 
More precisely, consider a one-parameter family of immersed curves $\gamma:\bar{I}\times[0,T]\to\mathbf{R}^2$ satisfying the following conditions: (i) $\gamma$ must lie above a given function $\psi$ called an obstacle, (ii) it evolves as the normal velocity coincides with the $L^2(\mathrm{d}s)$-gradient of $\mathcal{E}_\lambda$ as long as $\gamma$ does not touch an obstacle. 
We also impose the natural boundary condition, where (iii) the endpoints are fixed and (iv) the curvature at the both endpoints vanish. Then, the problem we consider is formally given by
\begin{align} \label{eq:P}\tag{P}
\begin{cases}
        \text{(i) \ \ } \gamma^2 \geq \psi\circ \gamma^1 \quad &\text{on} \ \ I\times [0,T], \\
        \text{(ii) \ } \partial_t^\perp\gamma = -\nabla_{L^2}\mathcal{E}_\lambda[\gamma] \quad &\text{on} \ \ \{(x,t) \mid x\in \mathcal{N}_t\}, \\
        \text{(iii) \,} \gamma(0,t)=(0,0), \ \ \gamma(1,t)=(1,0) \quad &\text{for all} \ \ t\in[0,T], \\
        \text{(iv) \,} \kappa(0,t)=\kappa(1,t)=0 \quad &\text{for all} \ \ t\in[0,T], 
\end{cases}
\end{align}
where $\partial_t^\perp\gamma$ denotes the normal velocity defined by $\partial_t^\perp\gamma:=\partial_t\gamma-\langle\partial_t\gamma, \partial_s \gamma\rangle\partial_s \gamma$, and $\mathcal{N}_t \subset I$ is the \emph{noncoincidence set} defined by 
\[\mathcal{N}_t:=\big\{x\in I \mid \gamma^2(x,t)>\psi(\gamma^1(x,t))\big\}\] 
for each $t\in[0,T]$.
Throughout this note we suppose that the given obstacle function \textcolor{black}{$\psi:\bar{I}\to\mathbf{R}$} satisfies
\begin{align}\label{eq:psi_condition}
\psi(0)<0, \ \psi(1)<0 \quad \text{and}\quad \psi\in C(\bar{I}).
\end{align}
An aim of this note is to show existence of solutions to \eqref{eq:P} among graphical curves $\gamma=\gamma_u(x,t)=(x,u(x,t))$ for some $u:\bar{I}\times [0,T]\to\mathbf{R}$, and hereafter we often identify $\gamma_u$ with $u$. 
The initial datum $\gamma_0(x)=(x,u_0(x))$ is chosen such that $u_0\in K:=\{ u\in H(I) \mid u \geq \psi \text{ in } I\}$, where $H(I):=H^2(I)\cap H^1_0(I)$ (see also Section~\ref{sect:preliminary}). 

Without the obstacle, existence of a classical solution of \eqref{eq:elastic_flow} under the natural boundary condition is already known (cf.\ \cite[Theorem 2.1]{NO_2014}).
On the other hand, it is not clear whether there exists a classical solution under the obstacle constraint since the loss of regularity of solutions may occur in the obstacle problem (see also Section~\ref{subsect:regularity}). 
Thus we have to introduce some suitable notion of a weak solution. 
\begin{definition}[Weak solution]\label{def:weak_sol}
Let $T>0$. We say that $u:\bar{I}\times[0,T]\to\mathbf{R}$ is a \emph{weak solution} of \eqref{eq:P} if $u$ satisfies the following conditions:
\begin{itemize}
    \item[(i)] $u$ lies in
    \begin{align}\label{eq:K_T}
    K_T:=\{ u\in L^\infty(0,T;H(I)) \cap H^1(0,T;L^2(I)) \mid 
        u \geq \psi \ \text{ in } {I}\times(0,T) \}; 
    \end{align}
    \item[(ii)] \ for all $v\in K_T$, 
    \begin{align}\label{eq:VI}
\begin{split}
    \int_0^T\!\!\int_I \bigg[\frac{\partial_tu (v-u)}{|\gamma_u'|} + 2\frac{u'' (v-u)''}{|\gamma_u'|^5} 
    &-5\frac{|u''|^2u'  (v-u)'}{|\gamma_u'|^7}  
     -\lambda \frac{u'' (v-u)}{|\gamma_u'|^3}  \bigg]\mathrm{d}x\mathrm{d}t  \geq0; 
\end{split}
\end{align} 
    \item[(iii)] \ \  the unique $C([0,T];L^2(I))$-representative of $u$ satisfies $u(0)=u_0$.
\end{itemize}
\end{definition}
Inequality \eqref{eq:VI} is called a \emph{variational inequality}.
We will provide a local-in-time existence result for weak solutions of \eqref{eq:P} via minimizing movements. 
\begin{theorem}[Existence of local-in-time solution]
\label{thm:existence}
Let $\lambda\geq0$ and suppose that $\psi \in C(\bar{I})$ satisfies \eqref{eq:psi_condition}. 
Then, for each $u_0 \in K$ there exists $T=T(u_0)>0$ such that \eqref{eq:P} possesses a weak solution $u$, which satisfies
\begin{align}\label{eq:thm_regularity}
    u \in L^{\infty}(0,T;H(I)) \cap H^1(0,T;L^{2}(I)) \cap L^{\frac{p}{p-1}}(0,T;W^{3,p}(I))
\end{align}
for any $p\in[2,\infty)$.
\end{theorem}
 
\vspace{-1.5\baselineskip}

\textcolor{black}{
\begin{remark}
Recently, Okabe and the author \cite{OY_2019} considered a very similar problem. 
Their variational inequality \cite[Equation P]{OY_2019} corresponds to the weak form of 
\[
\frac{d}{dt}\mathcal{E}_\lambda[\gamma_u(t)] = -\int_I |\partial_t u|^2\;\mathrm{d}x
\]
(see \cite[Equation 1.7]{OY_2019}). 
On the other hand, our inequality \eqref{eq:VI} comes from the weak form of
\begin{align}\label{eq:energy_structure}
\frac{d}{dt}\mathcal{E}_\lambda[\gamma_u(t)] = -\int_I \big|\partial_t^\perp\gamma_u \big|^2\;\mathrm{d}s
\end{align}
(see Section~\ref{subsect:bridge}).
Thus, the metric spaces considered in \cite{OY_2019} and in the present work are different. 
This difference yields a new nonlinearity that does not appear in \cite{OY_2019}.
Indeed, the variational inequality considered in \cite[Equation P]{OY_2019} corresponds to that obtained by replacing $\partial_t u/|\gamma_u'|$ with $\partial_t u$ in \eqref{eq:VI}. 
In this note, we rigorously verify the construction of solutions via a similar argument, where subtle differences arise; for instance, the argument requires the use of additional interpolation techniques.
\end{remark}
}

\textcolor{black}{
Another aim of this note is to confirm that if a weak solution $u$ of \eqref{eq:P} is smooth, then $\gamma_u$ can retrieve all the properties in \eqref{eq:P} and the energy identity \eqref{eq:energy_structure}. 
Equation \eqref{eq:energy_structure} is a well-known property of the classical solutions of \eqref{eq:elastic_flow} (see e.g.\ \cite[Lemma A.2]{DP14}).
}
On the other hand, it is not straightforward to deduce \eqref{eq:energy_structure} for weak solutions constructed via minimizing movements. 
To ensure that \eqref{eq:energy_structure} is well defined, we require $C^1(0,T;X)$, where $X$ denotes the natural energy space.
However, this regularity cannot be derived from a standard argument using minimizing movements.
Thus, we need a weaken version of \eqref{eq:energy_structure}. Indeed, in other geometric flows, some estimates related to \eqref{eq:energy_structure} have been considered: 
Laux--Otto \cite[Lemma 1]{LO20} deduced the estimate including the metric slope for the approximate solution of mean curvature flow; Okabe--Wheeler \cite[Equation 4]{OW23} obtained an integral form of \eqref{eq:energy_structure} (where an extra factor term $\tfrac{1}{2}$ appears on the right-hand side and the equality is replaced with the inequality) in the so-called $p$-elastic flow.

\textcolor{black}{ 
The other aim of this note is to provide an overview of existing results concerning obstacle problems for the bending energy. 
To this end, in Section~\ref{sect:survey} we briefly refer to several relevant works and summarize their main contributions.
In Subsection~\ref{subsect:regularity} we also refer to the existing results concerning the loss of regularity.
}

This paper is organized as follows:
\textcolor{black}{In Section~\ref{sect:survey} we collect and present several results in recent years about obstacle problems for the bending energy. 
}
In Section~\ref{sect:preliminary} we check that weak solutions $u$ of \eqref{eq:P} satisfy both all the conditions in \eqref{eq:P} and \eqref{eq:energy_structure}, under smoothness assumption.
We also observe from the existing literature that the loss of regularity of a weak solution of \eqref{eq:P} may occur. 
In Section~\ref{sect:existence}, we show a local-in-time weak solution via the minimizing movements. 
In Section~\ref{sect:appendix_Jacobi}, we collect the definitions of the elliptic integrals and the Jacobian elliptic integrals.

\section{Overview of related work}\label{sect:survey}
To the author's knowledge, the obstacle problem for the bending energy $\mathcal{B}$ was first studied by Dall'Acqua--Deckelnick \cite{DD18}, among graphical curves. 
Here, under the identification of $\gamma_u$ with $u$, the functional $\mathcal{B}(u)$ for $u\in H^2(I)$ is defined by 
\begin{align}\label{eq:B-graph}
\mathcal{B}(u)=\mathcal{B}[\gamma_u]=\int_I\Big(\frac{u''}{(1+(u')^2)^{\frac{3}{2}}} \Big)^2\sqrt{1+(u')^2}\;\mathrm{d}x.
\end{align}
Due to the lack of convexity of $\mathcal{B}$, the existence and uniqueness of minimizers become more delicate compared to the classical convex setting, such as that treated in \cite{KSbook}. 

Dall'Acqua--Deckelnick considered the minimization problem of $\mathcal{B}$ among $K$. 
The condition $u\in K$ also means that the endpoints of graphical curves are fixed, which is called the pinned boundary condition. 
They gave some conditions on the obstacle that assure existence of minimizers, and also discussed regularity (the regularity will be revisited in Section~\ref{subsect:regularity}). 
M\"uller \cite{Mueller19} considered the minimization of $\mathcal{B}$ among the set of the so-called pseudographs with endpoints fixed. 
He also showed that if $\psi$ is a symmetric cone obstacle and $\psi(\frac{1}{2})>h_*$, defined by  
\begin{align}\label{eq:h_*}
    h_*:= \frac{2}{c_0} \simeq 0.83462, 
    \quad \text{where} \quad c_0:=\frac{\sqrt{\pi}}{2}\frac{\Gamma(\frac{3}{4})}{\Gamma(\frac{5}{4})}, 
\end{align}
then there is no minimizer of $\mathcal{B}$ in $K$. 
On the other hand, Miura \cite{Miura21} and the author \cite{Ysima} independently showed that the minimizer of $\mathcal{B}$ in $K_{\rm sym}:=\{u\in K \mid u=u(1-\cdot)\}$ is unique if a symmetric cone $\psi$ satisfies $\psi(\frac{1}{2})< h_*$, and there exists no minimizer in $K_{\rm sym}$ if $\psi(\frac{1}{2})\geq h_*$.
Thus, $h_*$ represents the threshold height for existence on symmetric minimizers under the pinned boundary condition. 
While in \cite{Miura21} and \cite{Ysima} symmetry is imposed to the admissible set, M\"{u}ller demonstrated in \cite[Section 6]{Mueller21} that minimizers of $\mathcal{B}$ in $K$ must be symmetric if $\psi$ is symmetry and $\psi|_{[0,\frac{1}{2}]}$ is increasing, and the size of $\psi$ is sufficiently small. 

These studies have been extended in various directions. 
One is related to a generalization of the functional: Dall'Acqua et.\ al.\ \cite{DMOY24} dealt with the minimization problem of  $\mathcal{B}_p[\gamma]:=\int_\gamma |\kappa|^p\;\mathrm{d}s$ ($p\in (1,\infty)$) in \textcolor{black}{$\{ u\in W^{2,p}(I)\cap W^{1,p}_0(I) \mid u \geq \psi \text{ in } I\}$}, and they obtained the counterpart of the optimal height $h_*$ to ensure the existence of (symmetric) minimizers.
Another is related to a generalization of the admissible set: M\"{u}ller and the author \cite{MYarXiv2504} considered the minimization problem for $\mathcal{E}_\lambda$ with $\lambda\geq0$ among planar (not necessarily graphical) curves pinned at the endpoints, and they investigated how $\lambda$ depends on the coincidence set of minimizers. 

Recently, the minimization problem of $\mathcal{B}$ under the clamped boundary condition is also considered. 
For instance, Grunau--Okabe \cite{GO23, GO25} considered minimizers of $\mathcal{B}$ in $K_{\rm clamp}:=\{u\in H^2_0(I) \mid u\geq \psi \text{ in } I\}$.
They found that the conditions to ensure existence of minimizers in $K_{\rm clamp}$ are rather different from the pinned case.
In the clamped case, the threshold height is given by
\[
h^*:= \frac{1}{2}\max_{z \in[0,\infty)} \frac{2+2(1+z^2)^{-\frac{1}{4}}}{c_0-G(z)} \simeq 1.1890, \quad \text{where} \quad G(z):=\int_0^z\frac{1}{(1+y^2)^{\frac{5}{4}}}\mathrm{d}y.
\]
More precisely, it is shown in \cite[Theorem 1.2]{GO23} that there exists no minimizer of $\mathcal{B}$ in $K_{\rm sym} \cap K_{\rm clamp}$ if $\psi(x)\geq h^*$ for some $x\in I$. 
On the other hand, in \cite[Theorem 1.3]{GO25} they showed this height is optimal: 
for any $\delta>0$ there exists an obstacle $\psi_\delta$ satisfying \eqref{eq:psi_condition} and $\psi_\delta(\frac{1}{2})<h^*-\delta$ 
such that a minimizer of $\mathcal{B}$ in $K_{\rm sym} \cap K_{\rm clamp}$ exists. 
The differences between the pinned case and the clamped case are also observed in the energy bound. 
In the pinned case, it follows from \cite[Lemma 2.4]{DD18} that $\inf\{\mathcal{B}(u) \mid  u\in K\} \leq c_0^2$ for any obstacle $\psi$ satisfying \eqref{eq:psi_condition}, while in the clamped case, Grunau--Okabe \cite[Theorem 3.2]{GO25} found a family $\{\psi_\varepsilon\}_{\varepsilon>0}$ such that a minimizer of $u_\varepsilon$ of $\mathcal{B}$ in $K_{\rm sym} \cap K_{\rm clamp}$ satisfies $\mathcal{B}(u_\varepsilon) \to \infty$ as $\varepsilon\to0$.
In \cite{GO23}, they also applied their results to the studies on surfaces of revolution. 

Other obstacle problems involving the bending energy have been studied. 
Miura \cite{Miura16} dealt with the minimization problem for $\varepsilon^2\mathcal{B}+\mathcal{L}$ with the adhesion term under an obstacle constraint, and considered the limit $\varepsilon\to0$.
Dayrens--Masnou--Novaga \cite{DNP20} studies the problem of minimizing $\mathcal{B}$ among Jordan curves
confined in a given open set.
They showed existence of minimizers as well as some structural properties, e.g.\ a minimal curve has at least two contact points with the confinement.

We close this section by mentioning dynamical problems for $\mathcal{E}_\lambda$ with the obstacle constraint, which is particularly relevant to this manuscript. 
To the author's knowledge, a pioneering work on the parabolic obstacle problem for the bending energy is due to M\"{u}ller \cite{mueller20}. 
He studied the obstacle problem for the elastic flow in the metric space endowed with the $H(\mathrm{d}x)$-norm, not $L^2(\mathrm{d}s)$-norm (recall $H:=H^2\cap H^1_0$), so that solutions of the problem in \cite{mueller20} satisfy
\begin{align}\label{eq:energy_structure_H^2}
    \frac{d}{dt}\mathcal{E}_\lambda[\gamma_u(t)] = - \|\partial_t u(\cdot, t)\|_{H(I)} \quad \text{for a.e. } t\in(0,T),
\end{align}
instead of \eqref{eq:energy_structure} (see \cite[Proposition 2.6]{mueller20}). 
There are advantages to consider the $H(\mathrm{d}x)$-norm. 
One such advantage concerns regularity: in his setting, one may impose $H^1(0,T;H^2(I))$ to weak solutions as a natural class, so the regularity gets finer than our class $K_T$. 
This in particular allows us to obtain not only \eqref{eq:energy_structure_H^2} but also the property called ``energy dissipation'' as in \cite[Proposition 2.18]{mueller20}.
Another advantage is that the energy $\mathcal{E}_\lambda$ gets locally semiconvex. 
When the energy is (semi-)convex, the evolution admits a formulation as a metric gradient flow. 
Since the bending energy is not semiconvex with respect to the $L^2$-metric, employing the $H$-norm enables us to circumvent the difficulties caused by the lack of convexity. 

Parabolic obstacle problems for $\mathcal{E}_\lambda$ in the $L^2$-metric are also considered. 
Okabe and the author \cite{OY_2019} and M\"{u}ller \cite{Mueller21} have studied a variational inequality obtained by replacing $|\gamma_u'|^{-1} \partial_t u$ with $\partial_t u$ in inequality \eqref{eq:VI}. 
It is shown that a (unique) solution $u$ of their problem satisfies 
\[
\mathcal{E}_\lambda[\gamma_u(T)] -\mathcal{E}_\lambda[\gamma_u(0)] \leq -\frac{1}{2} \int_0^T \|\partial_t u(\cdot, t)\|_{L^2(I)}\;\mathrm{d}t
\]
(see e.g.\ \cite[Theorem 1.1]{OY_2019} and \cite[Lemma 4.17]{Mueller21}). 
Due to the lack of convexity, it is more difficult to improve the above estimate (e.g.\ to remove the factor $\frac{1}{2}$). 
This means that, compared to the $H^2$-framework, it is not easy to show that weak solutions can recover the same energy dissipation properties as those in \eqref{eq:energy_structure} or \eqref{eq:energy_structure_H^2}. 

\section{From weak to classical solutions}\label{sect:preliminary}

Hereafter, we always let $\lambda\geq0$ denote an arbitrary parameter.
Let $H(I)$ be the Hilbert space $H^2(I)\cap H^1_0(I)$ equipped with the inner product
$$
(u,v)_{H(I)} := \int_I u''v'' \;\mathrm{d}x \quad \text{for}\quad u, v \in H(I). 
$$
Throughout this paper we employ the norm $\Vert\cdot\Vert_{H(I)}$ on $H(I)$ as
$$
\Vert u\Vert_{H(I)} := (u,u)_{H(I)}^{1/2} \quad\text{for}\quad u \in H(I), 
$$
which is equivalent to $\Vert\cdot\Vert_{H^2(I)}$ (see e.g.\ \cite[Theorem 2.31]{GGS_2010}), i.e.\ there exists a positive constant $c_H$ such that 
\begin{equation}
\label{Geq:2.1} 
c_H \Vert u \Vert_{H^2(I)} \le \Vert u \Vert_{H(I)} \le \Vert u \Vert_{H^2(I)}.  
\end{equation}

\subsection{Remarks on the weak solution}\label{subsect:rem_weak_sol}

To begin with, note that the class $X:=L^\infty(0,T;H(I)) \cap H^1(0,T;L^2(I))$ in the admissible set $K_T $ in \eqref{eq:K_T} is slightly stronger than the natural class. 
Indeed, to define the weak form \eqref{eq:VI} it suffices to impose $u\in L^2(0,T;H(I)) \cap H^1(0,T;L^2(I))$. 
An advantage to introduce $X$ is to allow weak solutions to possess a continuous representation on $[0,T]$ with values in $C^1(\bar{I})$ by the embedding $X \subset C([0,T]; C^1(\bar{I}))$ (see e.g.\ \cite[Theorem II.5.16]{BF_2013}). 
Therefore, for each weak solution $u$ of \eqref{eq:P} the noncoincidence set 
\[\mathcal{N}_t:=\Set{x\in I | u(x,t)>\psi(x)}\] 
makes sense for every $t\in [0,T]$. 
In addition, we can use the admissible set $X$ without issue in our problem: in general, solutions constructed via the minimizing movements scheme often lie in $X$, and this is also the case in our problem (see Section~\ref{sect:existence}). 

In the variational inequality \eqref{eq:VI}, the test function $v\in K_T$ is taken to belong to the same convex set as the solution, in analogy with the classical framework presented in \cite{Fri_book}. 
However, as demonstrated in the following proposition, this test function framework can be refined.

\begin{proposition} \label{prop:Radon}
Let $u\in K_T$ be a weak solution of \eqref{eq:P}. Then, there exists $S\subset [0,T]$ such that $[0,T]\setminus S$ is a set of measure zero and 
\begin{align}\label{eq:Radon}
\int_I \bigg[\frac{\partial_tu (v-u)}{|\gamma_u'|}&+ 2\frac{u''(v-u)''}{|\gamma_u'|^{5}} -5\frac{|u''|^2u' (v-u)'}{|\gamma_u'|^7} -\lambda \frac{u'' (v-u)}{|\gamma_u'|^3} \bigg]\mathrm{d}x \geq 0 
\end{align}
for all $t\in S$ and $v\in K$. 
\end{proposition}
\begin{proof}
Fix $w\in K$ arbitrarily and set 
\[
f_w(t):=\int_I \bigg[\frac{\partial_tu}{|\gamma_u'|} (w-u) + 2\frac{u'' (w-u)''}{|\gamma_u'|^5}  -5\frac{|u''|^2u' (w-u)'}{|\gamma_u'|^7} -\lambda \frac{u'' (w-u)}{|\gamma_u'|^3}\bigg]\mathrm{d}x, 
\]
which makes sense for a.e.\ $t\in(0,T)$. 
In addition, from the fact that $u\in K_T$ we have $f_w \in L^1(0,T)$.
Let $\tau \in(0,T)$ and $r>0$ such that $(\tau-r,\tau+r)\subset (0,T)$. We claim that
\begin{align}
\begin{split}\label{eq:VI_before_Lebesgue}
J:=\frac{1}{2r}\int_{\tau-r}^{\tau+r}\int_I \bigg[\frac{\partial_tu (w-u)}{|\gamma_u'|} + 2\frac{u'' (w-u)''}{|\gamma_u'|^5}  -5\frac{|u''|^2u' (w-u)'}{|\gamma_u'|^7}
-\lambda \frac{u'' (w-u)}{|\gamma_u'|^3} & \bigg]\mathrm{d}x\mathrm{d}t \geq0.
\end{split}
\end{align}
Set $\tau_1:=\tau-r$ and $\tau_2:=\tau+r$. 
For $\varepsilon>0$, define $\phi_\varepsilon \in H^1(0,T;[0,1])$ by
\begin{align*}
\phi_\varepsilon(t):=
\begin{cases}
\frac{1}{\varepsilon}(t-\tau_1) +1 \ \ &\text{if } \ t\in[\tau_1-\varepsilon,\tau_1], \\
1 &\text{if } \ t\in[\tau_1,\tau_2], \\
-\frac{1}{\varepsilon}(t-\tau_2) +1 \ \ &\text{if } \ t\in[\tau_2,\tau_2+\varepsilon], \\
0 &\text{otherwise}.
\end{cases}
\end{align*}
Note that since $0\leq \phi_\varepsilon(t)\leq 1$, we have $\phi_\varepsilon w + (1-\phi_\varepsilon)u \in K_T$. 
Thus, by choosing $v=\phi_\varepsilon w + (1-\phi_\varepsilon)u$ in \eqref{eq:VI}, we obtain 
\begin{align}\label{eq:J_sum}
    J+J_1+J_2\geq0, 
\end{align} 
where $J_1:=\frac{1}{2r}\int_{\tau_1-\varepsilon}^{\tau_1}f_w(t)\;\mathrm{d}t$, and $J_2:=\frac{1}{2r}\int_{\tau_2}^{\tau_2+\varepsilon}f_w(t)\;\mathrm{d}t$. 
Since $f_w \in L^1(0,T)$, we have $J_1, J_2 \to 0$ as $\varepsilon\to0$. 
Therefore, passing the limit $\varepsilon\to0$ in \eqref{eq:J_sum}, we obtain \eqref{eq:VI_before_Lebesgue}. 

Now we turn to \eqref{eq:Radon}. 
For each $n\in\mathbf{N}$ we define $X_{T,n}\subset[0,T]$ by 
\begin{align*}
    X_{T,n}:= \Big\{ t\in[0,T] \,\Big|\, \limsup_{r\to 0} \Big|\frac{1}{2r} \int_{t-r}^{t+r}f_w(\tau)\;\mathrm{d}\tau -f_w(t)\Big| \geq \frac{1}{n} \Big\}. 
\end{align*}
Since $u\in L^\infty(0,T;H^2(I)) \cap H^1(0,T;L^2(I))$, by density for each $n\in\mathbf{N}$ there exist $\zeta_n \in C([0,T];H^2(I))$ and $\xi_n \in C([0,T];L^2(I))$ such that 
\begin{align}\label{eq:Bochner_dense}
\|u-\zeta_n\|_{L^2(0,T;H^2(I))} < \frac{1}{n2^n} \quad\text{and}\quad \|\partial_tu-\xi_n\|_{L^2(0,T;L^2(I))} < \frac{1}{n2^n}.
\end{align}
Then, from the fact that $u\in C([0,T];C^1(\bar{I}))$, we find that $g_w:[0,T]\to\mathbf{R}$ defined by
\[
g_w(t):=\int_I \bigg[\frac{\xi_n}{|\gamma_u'|}(w-u) + 2\frac{\zeta_n'' (w-\zeta_n)''}{|\gamma_u'|^5}  -5\frac{|\zeta_n''|^2u' (w-u)'}{|\gamma_u'|^7} -\lambda \frac{\zeta_n'' (w-u)}{|\gamma_u'|^3}\bigg]\mathrm{d}x
\]
is continuous on $[0,T]$.
The continuity of $g_w$ implies that
\[
\lim_{r\to0} \frac{1}{2r}\int_{t-r}^{t+r}g_w(\tau)\;\mathrm{d}\tau =g_w(t) \quad \text{for all} \ \ t\in(0,T). 
\]
Moreover, a straightforward calculation yields 
\begin{align*}
   |f_w(t)-g_w(t)| \leq C_* \Big(\|u(\cdot,t)-\zeta_n(\cdot,t)\|_{H^2(I)} &+ \|u(\cdot,t)-\zeta_n(\cdot,t)\|_{H^2(I)}^2\\
   &+\|\partial_tu(\cdot,t)-\xi_n(\cdot,t)\|_{L^2(I)}\Big)
\end{align*}
with a constant $C_*>0$ depending only on $T$, $\|u\|_{L^\infty(0,T;H^2(I))}$, and $\|w\|_{H^2(I)}$.
Now define the function $h:[0,T]\to\mathbf{R}_{\geq0}$ by
\[
h(t):=\|u(\cdot,t)-\zeta_n(\cdot,t)\|_{H^2(I)}+\|u(\cdot,t)-\zeta_n(\cdot,t)\|_{H^2(I)}^2+\|\partial_tu(\cdot,t)-\xi_n(\cdot,t)\|_{L^2(I)}. 
\]
Note that $h$ extended by $0$ outside $[0,T]$ belongs to $L^1_{\rm loc}(\mathbf{R})$. 
Using this $h$ we have
\begin{align*}
    \bigg|\frac{1}{2r}& \int_{t-r}^{t+r}f_w(\tau)\;\mathrm{d}\tau -f_w(t)\bigg| \\ 
    &\leq \frac{1}{2r} \int_{t-r}^{t+r}|f_w(\tau)-g_w(\tau)|\;\mathrm{d}\tau + \bigg|\frac{1}{2r} \int_{t-r}^{t+r}g_w(\tau)\;\mathrm{d}\tau -g_w(t)\bigg| +|f_w(t)-g_w(t)| \\
    &\leq \frac{1}{2r} \int_{t-r}^{t+r}C_*h(\tau)\;\mathrm{d}\tau + \bigg|\frac{1}{2r} \int_{t-r}^{t+r}g_w(\tau)\;\mathrm{d}\tau -g_w(t)\bigg| +C_*h(t), 
\end{align*}
from which it follows that 
\[
\limsup_{r\to0} \bigg|\frac{1}{2r} \int_{t-r}^{t+r}f_w(\tau)\;\mathrm{d}\tau -f_w(t)\bigg| 
\leq C_*\mathcal{M}(h)(t) +C_*h(t),
\]
where $\mathcal{M}(h)$ denotes the Hardy--Littlewood maximal function of $h$, i.e.
\[\mathcal{M}(h)(t):=\sup_{r>0}\frac{1}{2r}\int_{t-r}^{t+r}h(\tau)\;\mathrm{d}\tau.\] 
Hence, for every $n\in\mathbf{N}$, we have
\begin{align}\label{eq:est_X_Tn_depend_on_v}
X_{T,n} \subset \Set{ t\in[0,T] | \mathcal{M}(h)(t) >\tfrac{1}{2C_*n} } \cup \Set{ t\in[0,T] | h(t) >\tfrac{1}{2C_*n} }. 
\end{align}
Now we write 
\[
Y_{T,n}:=\Set{ t\in[0,T] | \mathcal{M}(h)(t) >\tfrac{1}{2C_*n} } \ \ \text{and} \ \  Z_{T,n}:=\Set{ t\in[0,T] | h(t) >\tfrac{1}{2C_*n} }, 
\]
both of which are independent of $w\in K$.
Recall from \eqref{eq:Bochner_dense} that 
$
\|h\|_{L^1(\mathbf{R})} \leq {C}/{n2^n}
$. 
Combining this with the weak maximal inequality for $L^1$-functions, we have 
\begin{align}\label{eq:est_Y_Tn}
|Y_{T,n}| \leq 2 C n \|h\|_{L^1(\mathbf{R})} \leq \frac{C}{2^n}.     
\end{align}
On the other hand, by Chebyshev's inequality it follows that 
\begin{align}\label{eq:est_Z_Tn}
|Z_{T,n}| \leq 2 C n \|h\|_{L^1(\mathbf{R})} \leq \frac{C}{2^{n}}. 
\end{align}
Noting that $\tilde{X}_{T,n}:=Y_{T,n} \cup Z_{T,n}$ is increasing with respect to $n$, we have
\[
 \bigcap_{n=1}^\infty \bigcup_{n\geq k}^\infty \tilde{X}_{T,k} = \bigcup_{n\in\mathbf{N}}\tilde{X}_{T,n} =: \tilde{X}_{T}. 
\]
As $\sum_{n=1}^\infty|\tilde{X}_{T,n}|<\infty$ holds by \eqref{eq:est_Y_Tn} and \eqref{eq:est_Z_Tn}, we may apply Borel--Cantelli lemma and conclude that $|\tilde{X}_{T}|=0$. 
Since \eqref{eq:est_X_Tn_depend_on_v} yields $\bigcup_{n\in\mathbf{N}} X_{T,n} \subset \tilde{X}_{T}$, we deduce from  definition of $X_{T,n}$ that \eqref{eq:Radon} holds for all $t\in (0,T)\setminus \tilde{X}_{T}=:S$.
Moreover, $S$ is independent of $w\in K$ as is $\tilde{X}_{T}$. 
 \end{proof}

\subsection{Consistency under smoothness assumptions}\label{subsect:bridge}

The aim of this subsection is to discuss whether a weak solution can retrieve all the conditions in \eqref{eq:P} as well as the same properties as given in \eqref{eq:energy_structure}. 
Throughout this subsection we suppose that a weak solution $u$ of \eqref{eq:P} is smooth and write $\gamma_u(x,t)=(x,u(x,t))$. 
Condition (i) in Definition~\ref{def:weak_sol} automatically implies (i) and (iii) in \eqref{eq:P}.
Note that if a weak solution $u$ is smooth, then $u$ satisfies \eqref{eq:Radon} for all $t\in(0,T]$.  

Hereafter fix $t\in(0,T]$ arbitrarily.
As a preliminary step, we show that
\begin{align}\label{eq:E-Leq_above_obst}
\int_I \bigg(\frac{\partial_tu}{|\gamma_u'|}+2\frac{1}{|\gamma_u'|}\Big(\frac{k_u'}{|\gamma_u'|}\Big)' + k_u^3 -\lambda k_u \bigg)\varphi\;\mathrm{d}x \ \textcolor{black}{=0} \quad \text{for all} \ \varphi \in C^\infty_{\rm c}(\mathcal{N}_t). 
\end{align}
Fix an arbitrary $\varphi \in C^\infty_{\rm c}(\mathcal{N}_t)$. 
If $\varepsilon_0>0$ is sufficiently small, then we find that $w=u(\cdot,t)+\varepsilon \varphi \in K$ for any $\varepsilon \in (-\varepsilon_0, \varepsilon_0)$. 
Thus \eqref{eq:Radon} implies that 
\[
\varepsilon\int_I \bigg(\frac{\partial_tu \varphi}{|\gamma_u'|} + 2\frac{u'' \varphi''}{|\gamma_u'|^5} -5\frac{|u''|^2u' \varphi'}{|\gamma_u'|^7} -\lambda \frac{u'' \varphi}{|\gamma_u'|^3} \bigg)\mathrm{d}x \geq0. 
\]
This inequality is valid for all sufficiently small $\varepsilon$, both positive and negative, and so in fact 
\begin{align}\label{eq:E-Leq_above_obst_notyet}
\int_I \bigg(\frac{\partial_tu \varphi}{|\gamma_u'|} + 2\frac{u'' \varphi''}{|\gamma_u'|^5} -5\frac{|u''|^2u' \varphi'}{|\gamma_u'|^7} -\lambda \frac{u'' \varphi}{|\gamma_u'|^3} \bigg)\mathrm{d}x =0.
\end{align}
Note that the signed curvature $k_u$ of $\gamma_u$ is given by $k_u=u''/|\gamma_u'|^3$.
Then integration by parts yields 
\begin{align}
\int_I \bigg(2\frac{u'' \varphi''}{|\gamma_u'|^5} -5\frac{|u''|^2u' \varphi'}{|\gamma_u'|^7} \bigg)\mathrm{d}x &= \int_I \bigg(2\frac{k_u\varphi''}{|\gamma_u'|^2} -5\frac{k_u^2 u' \varphi'}{|\gamma_u'|} \bigg)\mathrm{d}x \label{eq:250524-1}\\
&=\int_I \bigg(-2\frac{k_u' \varphi'}{|\gamma_u'|^2} -\frac{k_u^2 u' \varphi'}{|\gamma_u'|} \bigg)\mathrm{d}x \notag\\
&=\int_I \bigg(2\frac{1}{|\gamma_u'|}\Big(\frac{k_u'}{|\gamma_u'|}\Big)' -2\frac{u'' u'}{|\gamma_u'|^3}\frac{k_u'}{|\gamma_u'|} \notag \\ 
&\hspace{33pt}+ \frac{2k_u k_u' u'}{|\gamma_u'|}  + \frac{k_u^2 u''}{|\gamma_u'|} - \frac{k_u^2 u'' (u')^2}{|\gamma_u'|^3}\bigg)\varphi\;\mathrm{d}x \notag \\
&=\int_I \bigg(2\frac{1}{|\gamma_u'|}\Big(\frac{k_u'}{|\gamma_u'|}\Big)' + k_u^3 \bigg)\varphi\;\mathrm{d}x. \notag
\end{align}
Together with \eqref{eq:E-Leq_above_obst_notyet} this gives \eqref{eq:E-Leq_above_obst}. 

Here let us show that the curvature vector $\kappa_u$ of $\gamma_u$ satisfies \eqref{eq:P}-(iv), i.e.\
\begin{align}\label{eq:NavierBC_smooth}
\kappa_u(0,t)=\kappa_u(1,t)=0 \quad \text{for all} \ \ t\in(0,T]. 
\end{align}
Fix $t\in(0,T]$. 
Since $u(0,t)=0>\psi(0)$ and $u(\cdot,t), \psi \in C(\bar{I})$ by \eqref{eq:psi_condition}, there exists $\delta>0$ such that $u(x,t)>\psi(x)$ for all $x\in[0,\delta]$. 
Choose $\varphi \in C^2([0,\delta])$ such that $\varphi\geq0$, $\varphi(0)=0$, $\varphi'(0)=1$, and $\supp \varphi \subset[0,\delta)$. 
If $\varepsilon_0>0$ is sufficiently small, then we find that $w=u(\cdot,t)+\varepsilon \varphi \in K$ for any $\varepsilon \in (-\varepsilon_0, \varepsilon_0)$. 
Thus \eqref{eq:Radon} implies that 
\[
\varepsilon\int_I \bigg(\frac{\partial_tu \varphi}{|\gamma_u'|} + 2\frac{u'' \varphi''}{|\gamma_u'|^5} -5\frac{|u''|^2u' \varphi'}{|\gamma_u'|^7} -\lambda \frac{u'' \varphi}{|\gamma_u'|^3} \bigg)\mathrm{d}x \geq0. 
\]
Since this inequality holds for both positive and negative $\varepsilon$, similar to \eqref{eq:E-Leq_above_obst_notyet} we have 
\begin{align}\label{eq:250528-1}
\int_I \bigg(\frac{\partial_tu \varphi}{|\gamma_u'|} + 2\frac{u'' \varphi''}{|\gamma_u'|^5} -5\frac{|u''|^2u' \varphi'}{|\gamma_u'|^7} -\lambda \frac{u'' \varphi}{|\gamma_u'|^3} \bigg)\mathrm{d}x =0. 
\end{align}
Following the same argument as in \eqref{eq:250524-1}, we compute 
\begin{align*}
\begin{split}
    & \int_I \bigg(2\frac{u'' \varphi''}{|\gamma_u'|^5} -5\frac{|u''|^2u' \varphi'}{|\gamma_u'|^7}  \bigg)\mathrm{d}x \\
    =&\int_{[0,\delta]} \bigg(2\frac{1}{|\gamma_u'|}\Big(\frac{k_u'}{|\gamma_u'|}\Big)'+k_u^3\bigg)\varphi\;\mathrm{d}x+\bigg[2\frac{k_u \varphi'}{|\gamma_u'|^2} -2\frac{k_u'\varphi}{|\gamma_u'|^2}-\frac{k_u^2 u'\varphi}{|\gamma_u'|}\bigg]_{x=0}^{x=\delta} \\
    =&\int_{[0,\delta]} \bigg(2\frac{1}{|\gamma_u'|}\Big(\frac{k_u'}{|\gamma_u'|}\Big)'+k_u^3\bigg)\varphi\;\mathrm{d}x+2\frac{k_u(0)}{|\gamma_u'(0)|^2}. 
\end{split}
\end{align*}
Thus \eqref{eq:250528-1} implies 
\begin{align*}\notag 
\int_{[0,\delta]} \bigg(\frac{\partial_tu}{|\gamma_u'|}+2\frac{1}{|\gamma_u'|}\Big(\frac{k_u'}{|\gamma_u'|}\Big)' + k_u^3 -\lambda k_u \bigg)\mathrm{d}x +2\frac{k_u(0)}{|\gamma_u'(0)|^2}=0. 
\end{align*}
In addition, noting that $(0,\delta] \subset \mathcal{N}_t$, we deduce from \eqref{eq:E-Leq_above_obst} that 
\[
\frac{\partial_tu}{|\gamma_u'|}+2\frac{1}{|\gamma_u'|}\Big(\frac{k_u'}{|\gamma_u'|}\Big)' + k_u^3 -\lambda k_u  =0 \quad \text{on} \ \ (0,\delta).
\]
As a consequence it follows that $k_u(0,t)=0$, which also means that $\kappa_u(0,t)=0$. 
Similarly we obtain $\kappa_u(1,t)=0$, and hence \eqref{eq:NavierBC_smooth} has been proven.

Next we check \eqref{eq:P}-(ii).  
Now we fix an arbitrary $\eta\in C^\infty_{\rm c}(\mathcal{N}_t;\mathbf{R}^2)$. 
With the choice of $\varphi=-u' \eta^1$ and $\varphi=\eta^2$ in \eqref{eq:E-Leq_above_obst}, we obtain
\begin{align*}
    \int_I \frac{\partial_tu}{|\gamma_u'|}\frac{(-u') \eta^1}{|\gamma_u'|} |\gamma_u'|\;\mathrm{d}x &= - \int_I \bigg[ 2\frac{1}{|\gamma_u'|}\Big(\frac{k_u'}{|\gamma_u'|}\Big)' + k_u^3 -\lambda k_u \bigg]\frac{(-u')\eta^1}{|\gamma_u'|} |\gamma_u'|\;\mathrm{d}x, \\
    \int_I \frac{\partial_tu}{|\gamma_u'|}\frac{\eta^2}{|\gamma_u'|} |\gamma_u'|\;\mathrm{d}x &= - \int_I \bigg[ 2\frac{1}{|\gamma_u'|}\Big(\frac{k_u'}{|\gamma_u'|}\Big)' + k_u^3 -\lambda k_u \bigg]\frac{\eta^2}{|\gamma_u'|} |\gamma_u'|\;\mathrm{d}x.
\end{align*}
Adding them, and taking into account that the unit normal vector $\nu$ is given by $\nu=|\gamma_u'|^{-1}(-u',1)$, we have 
\begin{align}\label{eq:psuedo-eq:P'}
    \int_I \frac{\partial_tu}{|\gamma_u'|}\langle \nu,\eta \rangle \;\mathrm{d}s = - \int_I \bigg[ 2\frac{1}{|\gamma_u'|}\Big(\frac{k_u'}{|\gamma_u'|}\Big)' + k_u^3 -\lambda k_u \bigg]\langle \nu,\eta \rangle \;\mathrm{d}s, 
\end{align}
where $\mathrm{d}s:=|\gamma_u'|\;\mathrm{d}x$. 
Noting that $\partial_t \gamma_u = (0,\partial_t u)$, we have
\begin{align}\label{eq:normal_velocity}
     \partial_t^\perp \gamma_u = \langle \partial_t \gamma_u, \nu \rangle \nu =\frac{\partial_tu}{|\gamma_u'|}\nu. 
\end{align}
On the other hand, recall from \cite[Lemma A.1]{DP14} that $L^2(\mathrm{d}s)$-gradient $\nabla_{L^2}\mathcal{E}_\lambda[\gamma]$ of $\mathcal{E}_\lambda$ at $\gamma:I\to\mathbf{R}$ is given by
\[
\big\langle \nabla_{L^2}\mathcal{E}_\lambda[\gamma], \eta\big\rangle 
=\int_I \big\langle 2\nabla_s^2\kappa+|\kappa|^2\kappa-\lambda \kappa, \eta\big\rangle\;\mathrm{d}s
=\int_I \big(2k_{ss}+k^3-\lambda k\big) \langle\nu, \eta\rangle\;\mathrm{d}s. 
\]
Therefore, \eqref{eq:psuedo-eq:P'} implies that 
\[
\big\langle \partial_t^\perp\gamma_u, \eta \big\rangle_{L^2(\mathrm{d}s)} = \big\langle -\nabla_{L^2}\mathcal{E}_\lambda[\gamma_u], \eta\big\rangle, \quad \text{for all}\ \ \eta \in C^\infty_{\rm c}(\mathcal{N}_t;\mathbf{R}^2), 
\]
where in the left-hand side $\langle f, g\rangle_{L^2(\mathrm{d}s)} := \int_I \langle f, g\rangle \;\mathrm{d}s$.
Thus in this sense we observe that $\gamma_u$ satisfies \eqref{eq:P}-(ii).

\begin{remark} \label{rem:normal_velocity}
In this remark, we discuss the relationship between equation \eqref{eq:P}-(ii) and the original equation \eqref{eq:elastic_flow} of elastic flow.
In general, if a smooth one-parameter family $\gamma:\bar{I}\times [0,T]\to\mathbf{R}^2$ satisfies $\partial_t^\perp\gamma=-\nabla_{L^2} \mathcal{E}_\lambda[\gamma]$, then, after a suitable reparametrization, the normal velocity $\partial_t^\perp\gamma$ can be identified with the full velocity $\partial_t \gamma$. 
While for more details we refer the reader to \cite[Section 2.4]{MPP21}, in the following we instead give the outline of the justification: 
let $\theta:=\langle \partial_t \gamma, \partial_s \gamma \rangle$ denotes the tangential velocity, and let $\phi:\bar{I}\times [0,T]\to \bar{I}$ be the smooth solution of 
\[
\begin{cases}
\partial_t \phi (x,t) = -\frac{\theta(\phi(x,t),t)}{|\partial_x \gamma(\phi(x,t),t)|} \\
\phi(x,0)=x
\end{cases}
\] 
(such $\phi$ exists by classical ODE theory). 
Then, if $\gamma$ is smooth and it is a solution of $\partial_t^\perp\gamma=-\nabla \mathcal{E}_\lambda[\gamma]$, we find that the reparametrization 
$ 
\bar{\gamma}(x,t) := \gamma(\phi(x,t),t) 
$ 
is a solution of \eqref{eq:elastic_flow} (see e.g.\ \cite[proof of Theorem 4.1]{RS24}).
This is based on the fact that the energy functional $\mathcal{E}_\lambda$ is invariant under reparametrization so its $L^2$-gradient lies in the normal direction. As a result, the tangential component of the velocity does not contribute to the $L^2$ inner product with the gradient, and hence does not affect the energy dissipation.

On the other hand, in our problem we consider a family of \emph{graphical} curves, where the normal velocity and the full velocity do not generally coincide under the original parameterization.
Based on this consideration, we use the normal velocity formulation \eqref{eq:P}-(ii) as the appropriate weak form of the elastic flow \eqref{eq:elastic_flow}.

In addition, due to the obstacle constraint it is not clear whether weak solutions of \eqref{eq:P} can retrieve the regularity which enables the above argument (see Section~\ref{subsect:regularity}). 
A similar situation arises in the so-called \emph{$p$-elastic flow}.
The $p$-elastic flow is interpreted as the $L^2$-gradient flow of $E_p:=\mathcal{B}_p+\lambda \mathcal{L}$, where $\mathcal{B}_p[\gamma]:=\int_\gamma |\kappa|^p\;\mathrm{d}s$ ($p>1$). 
Due to the nonlinearity of the curvature term, the regularity of solutions to the $p$-elastic flow is delicate.
For instance, to address regularity issues in the analysis of weak solutions, Blatt--Hopper--Vorderobermeier \cite{BHV} introduced a weak formulation of the $p$-elastic flow based on the normal velocity, i.e.\ a weak form of $(\partial_t \gamma)^\perp = -\nabla_{L^2} E_p[\gamma]$.
\end{remark}

Now we turn to \eqref{eq:energy_structure}.
It immediately follows that 
\begin{align*}
    \frac{d}{dt}\mathcal{L}[\gamma_u(t)] = \int_I \frac{u' \partial_tu'}{|\gamma_u'|}\;\mathrm{d}x = \bigg[\frac{u' \partial_t u}{|\gamma_u'|}\bigg]_{x=0}^{x=1} -\int_I k_u{\partial_tu}\;\mathrm{d}x. 
\end{align*}
On the other hand, similar to \eqref{eq:250524-1} we compute 
\begin{align*} 
\begin{split}
    \frac{d}{dt}\mathcal{B}[\gamma_u(t)] &= \frac{d}{dt}\int_I \frac{(u'')^2}{|\gamma_u'|^5}\;\mathrm{d}x 
    =\int_I \bigg(2\frac{u'' \partial_t u''}{|\gamma_u'|^5}-5\frac{(u'')^2u' \partial_t u'}{|\gamma_u'|^7}\bigg)\mathrm{d}x \\
    &=\int_I \bigg(2\frac{1}{|\gamma_u'|}\Big(\frac{k_u'}{|\gamma_u'|}\Big)'+k_u^3\bigg)\partial_t u\;\mathrm{d}x+\bigg[2\frac{k_u \partial_t u'}{|\gamma_u'|^2} -2\frac{k_u'\partial_t u}{|\gamma_u'|^2}-\frac{k_u^2 u'\partial_t u}{|\gamma_u'|}\bigg]_{x=0}^{x=1}.
\end{split}
\end{align*}
Therefore, we obtain 
\begin{align}\label{eq:250526-1}
\begin{split}
    \frac{d}{dt}\mathcal{E}_\lambda[\gamma_u(t)] &=\int_I \bigg(2\frac{1}{|\gamma_u'|}\Big(\frac{k_u'}{|\gamma_u'|}\Big)'+k_u^3 - \lambda k_u\bigg)\partial_t u\;\mathrm{d}x \\
    &\qquad +\bigg[2\frac{k_u \partial_t u'}{|\gamma_u'|^2} -2\frac{k_u'\partial_t u}{|\gamma_u'|^2}-\frac{k_u^2 u'\partial_t u}{|\gamma_u'|}-\lambda\frac{u' \partial_t u}{|\gamma_u'|}\bigg]_{x=0}^{x=1}.
\end{split}
\end{align}
By \eqref{eq:NavierBC_smooth} we have $k_u(0,t)=k_u(1,t)=0$, and since $u(\cdot,t) \in H^1_0(I)$ for each $t\in[0,T]$, it also follows that $\partial_t u|_{x=0}=\partial_t u|_{x=1}=0$. 
Thus the boundary terms in the right-hand side of \eqref{eq:250526-1} are zero. 
Further, for any $x\in I\setminus \mathcal{N}_t$ and for all $h>0$ we have $u(x,t+h)\geq \psi(x) = u(x,t)$, so that $\partial_t u(x,t) = \lim_{h\downarrow0}\frac{1}{h}(u(x,t+h)-u(x,t)) \geq0$.
Similarly it follows that $\partial_t u(x,t) = \lim_{h\uparrow0}\frac{1}{h}(u(x,t+h)-u(x,t)) \leq0$, and hence 
\begin{align}\label{eq:pd_t=0_noncoincidence}
    \partial_t u(x,t) =0 \quad \text{for all} \ \ x\in I\setminus \mathcal{N}_t. 
\end{align}
Thus \eqref{eq:250526-1} is reduced to
\[
\frac{d}{dt}\mathcal{E}_\lambda[\gamma_u(t)] =\int_{\mathcal{N}_t} \bigg(2\frac{1}{|\gamma_u'|}\Big(\frac{k_u'}{|\gamma_u'|}\Big)'+k_u^3 - \lambda k_u\bigg)\partial_t u\;\mathrm{d}x. 
\]
Together with \eqref{eq:E-Leq_above_obst} this yields
\begin{align*}
\frac{d}{dt}\mathcal{E}_\lambda[\gamma_u(t)] =\int_{\mathcal{N}_t} -\frac{\partial_t u}{|\gamma_u'|}\partial_t u\;\mathrm{d}x = -\int_{I} \big|\partial_t^\perp\gamma_u\big|^2\;\mathrm{d}s, 
\end{align*}
where in the last equality we used \eqref{eq:normal_velocity} and \eqref{eq:pd_t=0_noncoincidence}.
Thus \eqref{eq:energy_structure} is now proven.

\subsection{Loss of regularity}\label{subsect:regularity}
In contrast to the previous argument, it is not easy to ensure that weak solutions of \eqref{eq:P} have the regularity that enables the argument in Section~\ref{subsect:bridge}. 
It is already known that a loss of regularity may occur at the level of stationary solutions.
If $u=u_0\in K$ is a solution to the variational inequality 
\begin{align}\label{eq:VI_station}
\begin{split}
    \int_I \bigg[2\frac{u'' (v-u)''}{|\gamma_u'|^5} 
    &-5\frac{|u''|^2u'  (v-u)'}{|\gamma_u'|^7}  
     -\lambda \frac{u'' (v-u)}{|\gamma_u'|^3}  \bigg]\mathrm{d}x  \geq0  \quad \text{for all} \ \ v\in K,
\end{split}
\end{align}
then $u$ is a stationary solution of \eqref{eq:P}. 
In the case of $\lambda=0$, the regularity of solutions to the above variational inequality is well studied. 
Indeed, it is shown in \cite[Theorem 5.1]{DD18} that every solution to \eqref{eq:VI_station} can lie in $W^{3,\infty}(I)$. 
This regularity is optimal in the following sense: the regularity of the unique symmetric solution of \eqref{eq:VI_station} cannot be improved to $C^3({I})$ if $\psi$ is a symmetric cone obstacle (i.e.\ $\psi=\psi(1-\cdot)$ and $\psi|_{[0,\frac{1}{2}]}$ is affine linear) such that $\psi(\frac{1}{2})<h_*$ (cf.\ \cite[Theorem 1.2]{Ysima}). 
Thus, by choosing such an obstacle $\psi$ and the solution of \eqref{eq:VI_station} as an initial datum, we find that a weak solution $u\in K_T$ of \eqref{eq:P} \emph{never} satisfies $u(\cdot,t) \in C^3({I})$. 

We proceed to describe the loss of regularity in precise detail in the case of $\lambda=0$. 
If $u$ is a solution of \eqref{eq:VI_station}, then for all $(x_1,x_2) \subset \mathcal{N}:=\{x\in I \mid u(x)>\psi(x)\}$ we can show that $u\in C^\infty([x_1,x_2])$ (cf.\ \cite[Proposition 3.2]{DD18}). 
Moreover, by applying an argument similar to the one used to derive \eqref{eq:E-Leq_above_obst}, we obtain
\begin{align}\label{eq:VI_station_on_N}
    \int_I \bigg[2\frac{1}{|\gamma_u'|} \Big( \frac{k_u'}{|\gamma_u'|} \Big)' + k_u^3 \bigg]\varphi\;\mathrm{d}x  = 0  \quad \text{for all} \ \ \varphi \in C^\infty_{\rm c}(x_1,x_2).
\end{align}
Recalling that $\partial_s = |\gamma_u'|^{-1}\partial_x$, we infer from the above equality that $k_u$ satisfies $2(k_u)_{ss}+k_u^3=0$ above the noncoincidence set $\mathcal{N}$.
In general, a non-trivial curve whose signed curvature $k$ satisfies $2k_{ss}+k^3=0$ is called a \emph{rectangular elastica}, and up to similarity, its signed curvature is given by $k(s)=\sqrt{2}\alpha\cn(\alpha s+s_0, \frac{1}{\sqrt{2}})$ with some $\alpha>0$ and $s_0\in\mathbf{R}$ (see Section~\ref{sect:appendix_Jacobi} for definition of $\cn$).
Using this rectangular elastica, we can identify the unique symmetric solution $u$ of \eqref{eq:VI_station} as 
\begin{align}\label{eq:SCF}
\begin{cases}
        \text{(i) \ \ } \gamma_u \text{ is a part of a rectangular elastica on} \ (0,\frac{1}{2}), \\
        \text{(ii) \ \ } \text{the coincidence set } I\setminus \mathcal{N} = \{\frac{1}{2}\}, \\
        \text{(iii) \,} \text{the third weak derivative $u'''$ is discontinuous at } x=\frac{1}{2} \\
\end{cases}
\end{align}
(see Figure~\ref{fig:SCF}).

This characterization can be seen from the following outline: 
any solution of \eqref{eq:VI_station} must be concave, i.e.\ $u''(x)<0$ for $x\in I$ (see e.g.\ \cite[Theorem 1.3]{DMOY24}).
Combining this fact with the assumption that $\psi$ is a symmetric cone, we can deduce that the coincidence set $I\setminus \mathcal{N}$ must be equal to $\{\tfrac{1}{2}\}$. 
This yields property (ii) in \eqref{eq:SCF}. 
In addition, we can also deduce that the signed curvature $k_u$ of $\gamma_u$ must vanish at the both endpoints (see e.g.\ \cite[Theorem 1.3]{DMOY24}), and hence $k_u$ must satisfy $2\partial_s^2 k_u + k_u^3=0$ on $(0,\frac{1}{2})$ and $k_u(0)=0$. 
Therefore, recalling the concavity of $u$, we find that, up to dilation and reparametrization, $k_u$ is represented by 
\begin{align}\label{eq:rect_curvature}
k(s):=-\sqrt{2} \cn(s -\mathrm{K}(\tfrac{1}{\sqrt{2}}), \tfrac{1}{\sqrt{2}})
\end{align}
(see Section~\ref{sect:appendix_Jacobi} for more details). 
This yields property (i) in \eqref{eq:SCF}. 
Now let $\Gamma_{\rm rect}:[0,2\mathrm{K}(\tfrac{1}{\sqrt{2}})] \to \mathbf{R}^2$ be a curve with curvature  given by \eqref{eq:rect_curvature} such that $\Gamma_{\rm rect}(0)=0$ (see Figure~\ref{fig:SCF}).
With the help of the argument using the so-called polar tangential angle (cf.\ \cite[Section 3.2]{Miura21}), we find that if $\psi(\frac{1}{2})<h_*$, then up to rotation, dilation, and reparametrization, $\gamma_u|_{[0,\frac{1}{2}]}$ is represented by $\Gamma_{\rm rect}|_{[0,s_0]}$ for some $s_0\in(0,\mathrm{K}(\tfrac{1}{\sqrt{2}}))$ (see Figure~\ref{fig:SCF}).
Since \eqref{eq:rect_curvature} implies that $k'(s)<0$ for all $s\in(0,\mathrm{K}(\tfrac{1}{\sqrt{2}}))$, we have $\lim_{x\uparrow\frac{1}{2}}k_u'(x)<0$. 
This also means that $\lim_{x\uparrow\frac{1}{2}}u'''(x)<0$, and simultaneously means discontinuity of $u'''$ at $x=1/2$ --- by symmetry the third derivative of $u$ would vanish at $x=1/2$ if $u$ could lie in $C^3(I)$. 
Thus property (iii) in \eqref{eq:SCF} follows.

\begin{center}
    \begin{figure}[htbp]
      \includegraphics[scale=0.18]{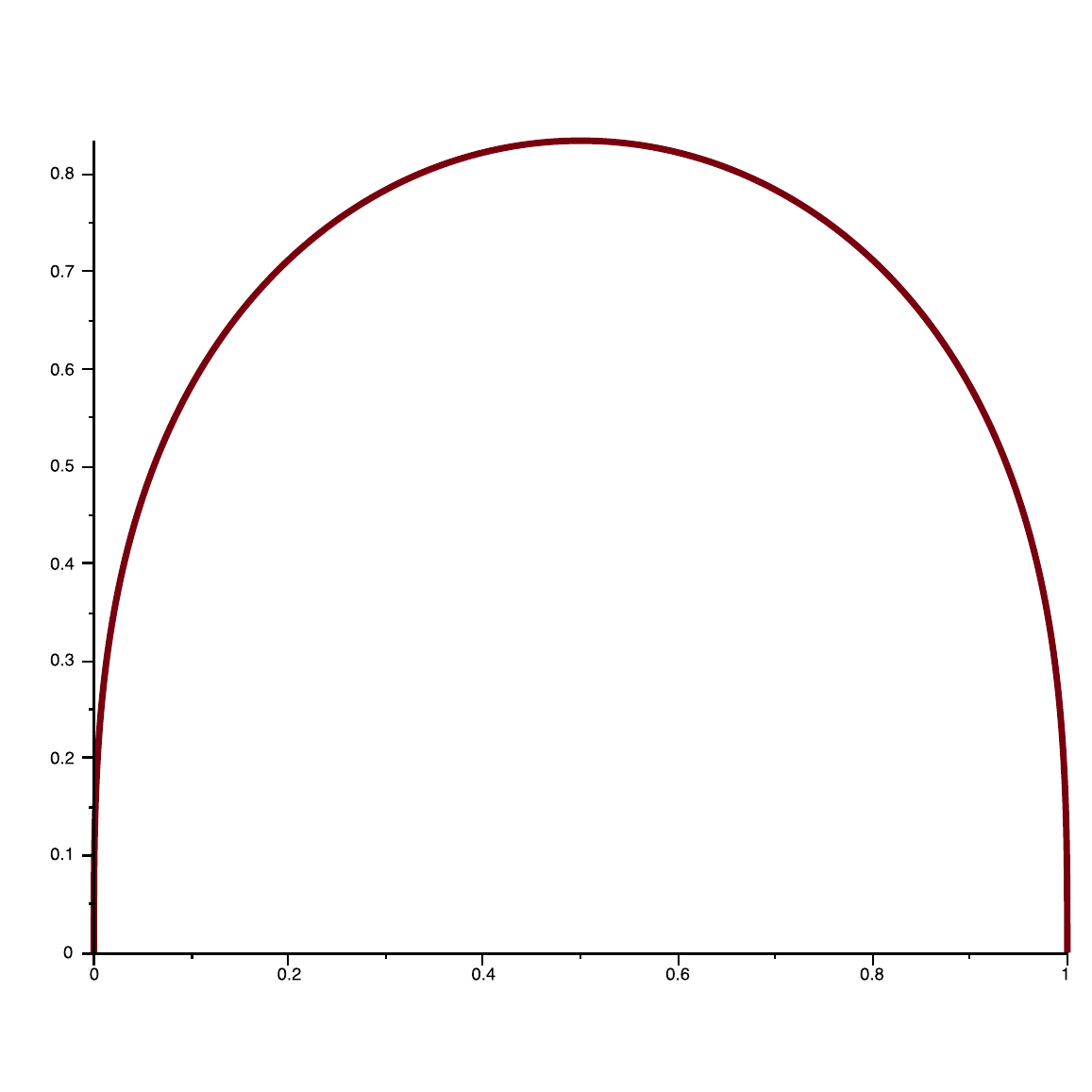}
      \quad
      \includegraphics[scale=0.18]{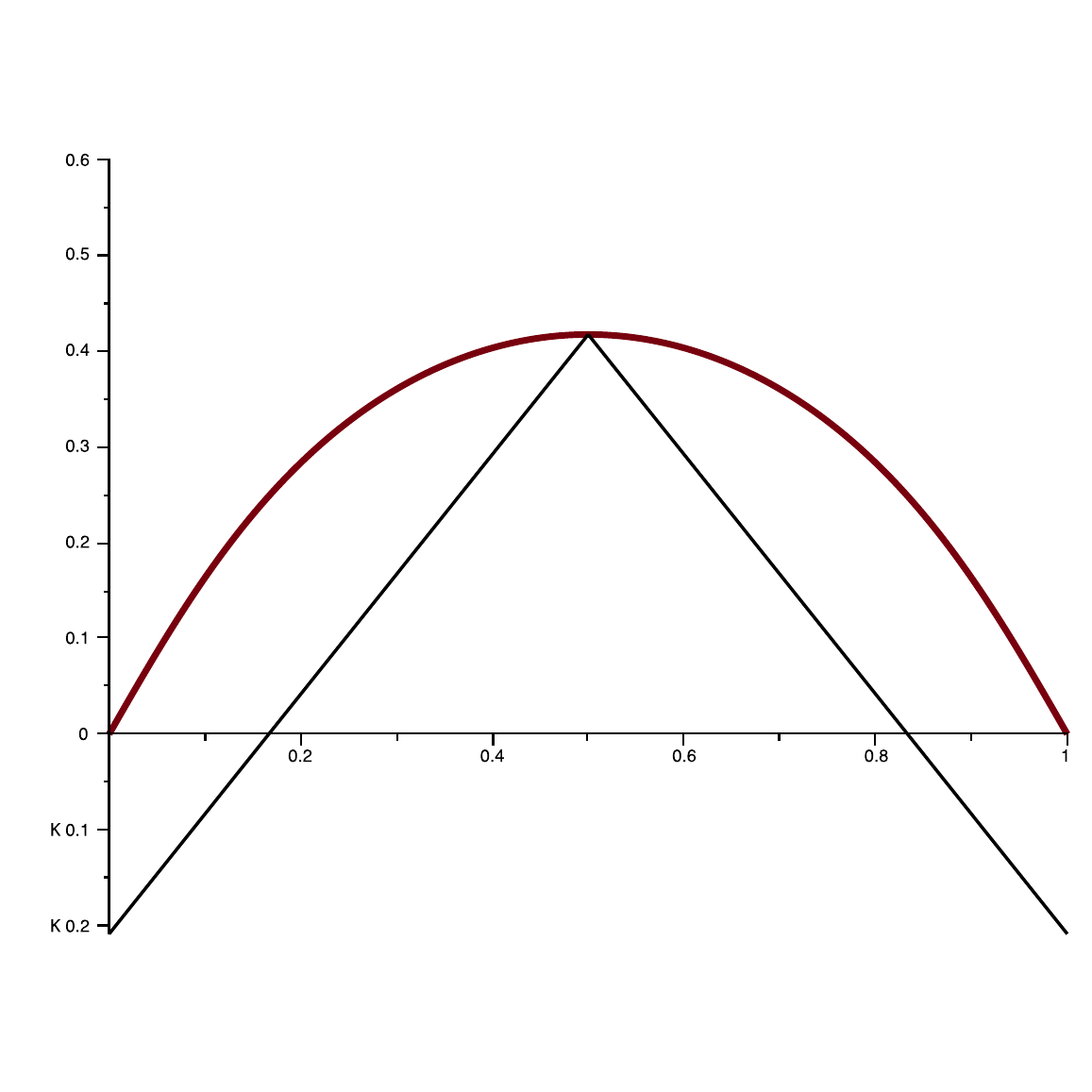}
      \quad
      \includegraphics[scale=0.18]{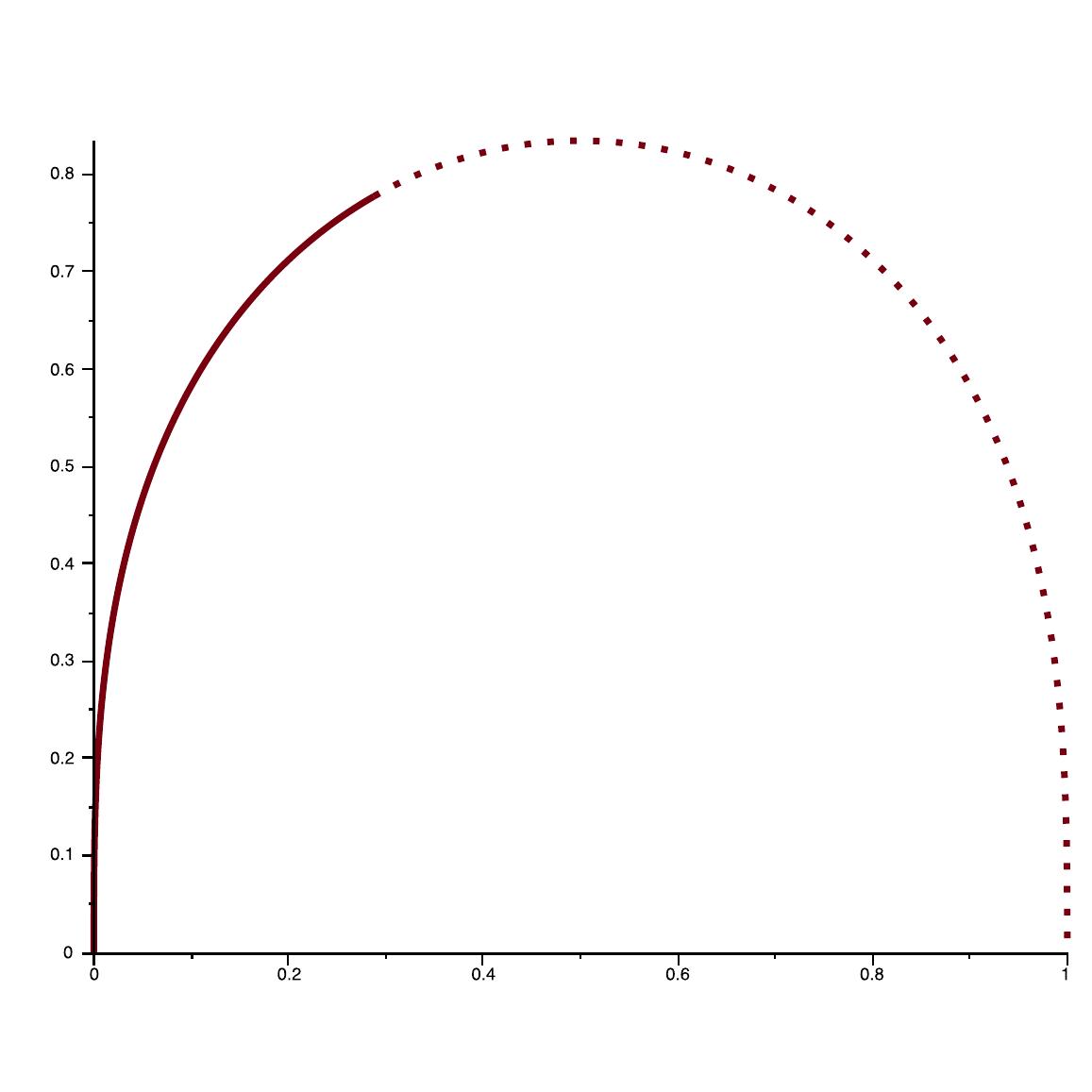} \\
  \caption{A rectangular elastica $\Gamma_{\rm rect}$ (left), the unique symmetric solution $u$ of \eqref{eq:VI_station} with a symmetric cone obstacle whose height is $h_*/2$. The third weak derivative of $u$ is discontinuous at $x=1/2$ (center). A part of $\Gamma_{\rm rect}$ which coincides with $u|_{[0,\frac{1}{2}]}$ in the center, after a suitable rotation and dilation (right).}
  \label{fig:SCF}
  \end{figure}
\end{center}

\section{Existence of a local-in-time solution}\label{sect:existence}
In this section we prove the existence of local-in-time weak solutions of \eqref{eq:P} via minimizing movements. 
We begin with some lemmas. 

\begin{lemma} \label{Gtheorem:2.1}
\begin{align}\label{eq:derivative-estimate}
\|f'\|_{L^\infty(I)}  \leq \sqrt{2}\Vert f\Vert_{L^2(I)}^{\frac{1}{4}}\Vert f\Vert_{H(I)}^{\frac{3}{4}} \quad \text{for}\quad f \in H(I).
\end{align}
In addition, $\|f\|_{L^\infty(I)} \leq \|f'\|_{L^\infty(I)}$ also holds for any $f \in H(I)$.
\end{lemma}
\begin{proof}
Since $f(0)=f(1)=0$, integration by part and H\"older's inequality yield 
\begin{equation}
\label{Geq:2.2}
\Vert f'\Vert^2_{L^2(I)} =-\int_I f\cdot f''\;\mathrm{d}x \le \Vert f\Vert_{L^2(I)} \Vert f''\Vert_{L^2(I)}. 
\end{equation}
Using that $f(0)=f(1)=0$ again, we find $x_{0} \in I$ such that $f'(x_0)=0$, and then 
$$
|f'(x)|^2 = \int_{x_0}^x \big(|f'(\xi)|^2\big)'\;\mathrm{d}\xi =\int_{x_{0}}^x 2f'f''\;\mathrm{d}\xi \leq 2\Vert f'\Vert_{L^2(I)} \Vert f''\Vert_{L^2(I)} 
$$
for all $x \in \bar{I}$. This together with \eqref{Geq:2.2} yields \eqref{eq:derivative-estimate}. 
The remaining estimate immediately follows from the fact that $f(x) = \int_0^xf'(y)\;\mathrm{d}y$. 
 \end{proof}

We will also use the following interpolation inequality. For the proof we refer the reader to \cite[Theorem~6.4]{FFLM12}.

\begin{proposition}\label{Gtheorem:2.2}
Let $\Omega \subset \mathbf{R}^N$ be a bounded open set satisfying the cone condition. 
Let $k$, $\ell$, and $m$ be integers such that $0 \le k \le \ell \le m$. 
Let $1 \le p \le q < \infty$ if $(m-\ell)p \ge N$ and let $1 \le p \le q \le \infty$ if $(m-\ell)p > N$. 
Then, there exists $C>0$ such that 
$$
\Vert D^\ell f \Vert_{L^q(\Omega)} \le C(\Vert D^m f \Vert^\alpha_{L^p(\Omega)} \Vert D^k f\Vert^{1-\alpha}_{L^p(\Omega)} + \Vert D^k f \Vert_{L^p(\Omega)}) 
$$
for all $f \in W^{m,p}(\Omega)$, where 
$$
\alpha := \dfrac{1}{m-k}\Bigl( \dfrac{N}{p} - \dfrac{N}{q} + \ell - k \Bigr). 
$$
\end{proposition}

\subsection{Construction of scheme} 

In the following we construct a family of approximate solutions via minimizing movements.
For $u\in K$ we define $\mathcal{E}_\lambda(u):=\mathcal{E}_\lambda[\gamma_u]$, so that 
\[\mathcal{E}_\lambda(u)=\int_I\frac{(u'')^2}{(1+(u')^2)^{\frac{5}{2}}}\;\mathrm{d}x + \lambda \int_I \sqrt{1+(u')^2}\;\mathrm{d}x
\]
(cf.\ \eqref{eq:B-graph}). 
We fix $u_0 \in K$ arbitrarily, and define 
\begin{align*}
M_0:=\max_{x \in \bar{I}} |u'_0(x)|, 
\end{align*}
and let $\rho=\rho(u_0, \lambda)>0$ be the constant defined by
\begin{align}\label{eq:def-rho}
\rho:=\max\left\{c_H^{-1} (1+4M_0^2)^{\frac{5}{4}}, \sqrt{2} \right\} \mathcal{E}_\lambda(u_0)^{\frac{1}{2}}. 
\end{align}
We define $T>0$ by a constant small enough to satisfy
\begin{align}\label{Geq:4.1}
\sqrt{2}(1+4M_0^2)^{\frac{1}{16}} \rho T^{\frac{1}{8}}\leq \frac{1}{2}M_0.
\end{align}
Under the above setting, we define the discrete approximate solutions. 
For $T>0$ defined by \eqref{Geq:4.1} and each $n\in \mathbf{N}$, set $
\tau_{n} := {T}/{n}
$.
We define inductively $\{ u_{i,n} \}^n_{i=0}$ as follows. 
Let $u_{0,n}:=u_0$ and for each $i=1, \ldots,n$, we define $u_{i,n}$ by a solution of the minimization problem
\begin{align} \label{Gelmin} \tag{M$_{i,n}$}
\min_{v \in \tilde{K}}  G_{i,n}(v), \quad  G_{i,n}(v) := \mathcal{E}_\lambda(v) + P_{i,n}(v),
\end{align}
where
\begin{align}\notag
P_{i,n}(v) := \frac{1}{2\tau_{n}}\int_I \frac{|v-u_{i-1,n}|^2}{\sqrt{1+(u_{i-1,n}')^2}} \;\mathrm{d}x
\quad \text{and}\quad 
\tilde{K}:= \Big\{ v \in K \mid  \max_{x \in \bar{I}} |v'(x)|\le 2M_0 \Big\}.
\end{align}
By a standard direct method combined with the fact that 
\[
\frac{1}{(1+4M_0^2)^{\frac{5}{2}}} \Vert u_{i,n}''\Vert_{L^2(I)}^2 
\le \int_I \frac{(u_{i,n}'')^2}{(1+|u_{i,n}'|^2)^{\frac{5}{2}}} \;\mathrm{d}x 
= \mathcal{B}(u_{i,n}), 
\]
we can show that problem \eqref{Gelmin} possesses a solution for any $n\in \mathbf{N}$ and each $i=1,\ldots,n$.
Here we define $w_{i,n}:I\to \mathbf{R}$ by 
$$
w_{i,n}(x) := \frac{u_{i,n}(x)-u_{i-1,n}(x)}{\tau_{n}}. 
$$
\begin{definition}\label{def:interpolant}
Let $n\in \mathbf{N}$ and $i\in\{1,\ldots,n\}$.
\begin{itemize}
\item[(i)] Let ${u_n}:\bar{I}\times[0,T]\to\mathbf{R}$ be the piecewise linear interpolation of ${u_{i,n}}$, i.e. 
\[
u_{n}(x,t):= \sum_{i=1}^n\chi_{[(i-1)\tau_{n},i\tau_{n}]}(t)\Big[ u_{i-1,n}(x) + \bigl( t-(i-1)\tau_{n} \bigr)w_{i,n}(x) \Big].
\]
\item[(ii)] We define $\bar{u}_{n}(x,t) :\bar{I}\times [0,T] \to \mathbf{R}$ and $\underline{u}_{n}(x,t) :\bar{I}\times [0,T] \to \mathbf{R}$ by piecewise constant interpolations of ${u_{i,n}}$ as follows: 
\begin{align*}
    \bar{u}_n(x,t):&= \sum_{i=1}^n\chi_{[(i-1)\tau_{n},i\tau_{n}]}(t) u_{i,n}(x), \quad
    \underline{u}_n(x,t):= \sum_{i=1}^n\chi_{[(i-1)\tau_{n},i\tau_{n}]}(t) u_{i-1,n}(x). 
\end{align*}
\end{itemize}
\end{definition}

Note that $u_n$ is piecewise smooth with respect to $t\in(0,T)$. Indeed, 
\begin{align}\label{eq:partial_t-u_n}
\partial_t u_n(x,t)=\sum_{i=1}^n\chi_{[(i-1)\tau_{n},i\tau_{n}]}(t)w_{i,n}(x) \quad \text{for }\ (x,t)\in I \times(0,T). 
\end{align}

We derive the following uniform estimates for $\{ u_{i,n} \}$ and $\partial_t u_n$. 
\begin{lemma} \label{Gtheorem:4.4}
Let $n \in \mathbf{N}$ and let $\rho>0$ be the constant given by \eqref{eq:def-rho}.
Then 
\begin{equation} 
\label{Geq:4.6}
 \sup_{0 \le i \le n} \Vert u_{i,n}\Vert_{H^2(I)}\leq \rho, \quad \int_0^T\!\!\!\int_I\frac{(\partial_t u_n)^2}{\sqrt{1+(\underline{u}_n')^2}}\;\mathrm{d}x\mathrm{d}t  \leq \rho^2. 
\end{equation}
\end{lemma}
\begin{proof}
Fix $n\in\mathbf{N}$ and $i\in\{1,\ldots,n\}$.
Since $u_{i,n}$ is a minimizer of $G_{i,n}$ in $\tilde{K}$, we have 
\begin{align}\label{Geq:4.7}
\mathcal{E}_\lambda(u_{i,n}) + P_{i,n}(u_{i,n}) =  G_{i,n}(u_{i,n})\leq  G_{i,n}(u_{i-1,n}) = \mathcal{E}_\lambda(u_{i-1,n}), 
\end{align}
where we used the fact that $u_{i-1,n} \in \tilde{K}$. 
This clearly implies that $\mathcal{E}_\lambda(u_{i,n}) \le \mathcal{E}_\lambda(u_0)$ for $i=1,\ldots,n$. 
Then, by the fact that $|u_{i,n}'|\leq 2M_0$ we obtain 
\begin{align*}
\frac{1}{(1+4M_0^2)^{\frac{5}{2}}} \Vert u_{i,n}''\Vert_{L^2(I)}^2 
\le \int_I \frac{(u_{i,n}'')^2}{(1+|u_{i,n}'|^2)^{\frac{5}{2}}} \;\mathrm{d}x 
= \mathcal{B}(u_{i,n}) 
\le \mathcal{E}_\lambda(u_{i,n}) 
\le \mathcal{E}_\lambda(u_{0}).
\end{align*}
This together with \eqref{Geq:2.1} implies that 
\begin{align}\label{Geq:4.8}
\sup_{0\leq i \leq n}\Vert u_{i,n}\Vert_{H^2(I)} \leq c_H^{-1}\sup_{0\leq i \leq n}\Vert u_{i,n}''\Vert_{L^{2}(I)} 
\le  c_H^{-1}(1+4M_0^2)^{\frac{5}{4}}\mathcal{E}_\lambda(u_{0})^{\frac{1}{2}}\leq \rho,
\end{align}
where the last inequality follows from the definition of $\rho$.

We turn to the estimate on $\partial_t u_n/\sqrt{1+(\underline{u}_n')^2}$. 
From \eqref{Geq:4.7} it follows that
$
P_{i,n}(u_{i,n})\leq\mathcal{E}_\lambda(u_{i-1,n}) - \mathcal{E}_\lambda(u_{i,n})
$. 
Then using \eqref{eq:partial_t-u_n} we compute 
\begin{equation}
\label{Geq:4.9}
\begin{split}
\int_0^T\!\!\!\int_I \frac{(\partial_t u_n)^2}{\sqrt{1+(\underline{u}_n')^2}} \;\mathrm{d}x\mathrm{d}t 
&= \sum_{i=1}^{n}\int_{(i-1)\tau_n}^{i\tau_n}\int_I \frac{w_{i,n}^2}{\sqrt{1+(u_{i-1,n}')^2}} \;\mathrm{d}x\mathrm{d}t 
= 2\sum_{i=1}^{n} P_{i,n}(u_{i,n}) \\
&\le 2\sum_{i=1}^{n} \Big( \mathcal{E}_\lambda(u_{i-1,n}) - \mathcal{E}_\lambda(u_{i,n}) \Big) 
 \le 2\mathcal{E}_\lambda(u_{0}) \leq \rho. 
\end{split}
\end{equation} 
The proof is complete.
 \end{proof}

The following lemma plays an important role in justifying the choice of the admissible set $\tilde{K}$ instead of $K$ in the minimization problem \eqref{Gelmin}. 

\begin{lemma} \label{Gtheorem:4.5}
For any $n\in\mathbf{N}$ and each $i=1,\ldots,n$, 
\begin{align*} 
\Vert u'_{i,n}\Vert_{L^{\infty}(I)} \leq \frac{3}{2}M_0. 
\end{align*}
\end{lemma}
\begin{proof}
Fix $0\leq t_1 < t_2 \leq T$ arbitrarily. 
By Lemma~\ref{Gtheorem:2.1} we have 
\begin{equation} 
\label{Geq:4.10}
\begin{split}
&\Vert u'_n(\cdot,t_2) - u'_n(\cdot,t_1) \Vert_{L^{\infty}(I)} \\
& \qquad \le \sqrt{2} \Vert u_n(\cdot,t_2)-u_n(\cdot,t_1) \Vert_{H(I)}^{\frac{3}{4}} \Vert u_n(\cdot,t_2)-u_n(\cdot,t_1) \Vert_{L^{2}(I)}^{\frac{1}{4}}.
\end{split}
\end{equation}
It follows from \eqref{Geq:2.1} and Lemma \ref{Gtheorem:4.4} that 
\begin{align} \label{Geq:4.11}
\sup_{t\in[0,T]} \Vert u_{n}(\cdot,t)\Vert_{H(I)} \le \sup_{t\in[0,T]} \Vert u_{n}(\cdot,t)\Vert_{H^2(I)}
 \le \sup_{0\leq i \le n} \Vert u_{i,n} \Vert_{H^2(I)} \le \rho.
\end{align}
By H\"older's inequality 
we see that  
\begin{align*}
\Vert u_n(\cdot,t_2)-u_n(\cdot,t_1) \Vert^2_{L^2(I)} 
& = \int_I \Big| \int_{t_1}^{t_2} \partial_t u_n\;\mathrm{d}t \Big|^2 \;\mathrm{d}x 
 \leq (t_2-t_1) \int_I \int_{t_1}^{t_2}|\partial_t u_n|^2\;\mathrm{d}t \mathrm{d}x \\
&\leq (t_2-t_1)\sqrt{1+4M_0^2} \int_I\int_{t_1}^{t_2}\frac{(\partial_t u_n)^2}{\sqrt{1+(\underline{u}_n')^2}}\;\mathrm{d}t \mathrm{d}x \\
& \le \sqrt{1+4M_0^2} (t_2 - t_1)\rho^2, 
\end{align*}
where in the last inequality we used Lemma~\ref{Gtheorem:4.4}.
Thus, by \eqref{Geq:4.10} and \eqref{Geq:4.11} we have 
\begin{align*}
\Vert u'_n(\cdot,t_2) - u'_n(\cdot,t_1) \Vert_{L^{\infty}(I)} 
\le \sqrt{2}(1+4M_0^2)^{\frac{1}{16}} \rho (t_2 - t_1)^{\frac{1}{8}}, 
\end{align*}
and this and \eqref{Geq:4.1} imply  
\begin{equation}
\label{Geq:4.13} 
\Vert u'_n(\cdot,t_2) - u'_n(\cdot,t_1) \Vert_{L^{\infty}(I)} \le \frac{1}{2} M_0. 
\end{equation}
In particular, for each $i=1, \ldots, n$, by taking $t_{2}=i\tau_{n}$ and $t_{1}=0$ in \eqref{Geq:4.13}, we have 
\begin{align*}
\Vert u'_{i,n}\Vert_{L^{\infty}(I)} \leq\Vert u'_n(\cdot,i\tau_{n})-u'_0 \Vert_{L^{\infty}(I)} + \Vert u'_0 \Vert_{L^{\infty}(I)}\leq \frac{1}{2} M_0 + M_0 = \frac{3}{2} M_0 . 
\end{align*}  
The proof is complete.
 \end{proof}

\subsection{Regularity} \label{Gsubsection:4.2}

In this subsection we discuss the regularity of approximate solutions with respect to the space variable. 
Fix a nonnegative function $\varphi\in C^{\infty}_{{\rm c}}(I)$ arbitrarily. 
Then it follows from Lemma~\ref{Gtheorem:4.5} that $u_{i,n} + \varepsilon \varphi \in \tilde{K}$ for $i=1, \ldots, n$ and $\varepsilon>0$ small enough. 
Hence, by the minimality of $u_{i,n}$ we have 
$$
\frac{d}{d\varepsilon} G_{i,n}(u_{i,n}+\varepsilon \varphi ) \Big|_{\varepsilon=0} \ge 0.
$$
The inequality is reduced into 
\begin{align} 
\begin{split}
\label{Geq:4.14}
L_{i,n}(\varphi):= \int_I\bigg[ 2\frac{u_{i,n}''\varphi''}{(1+(u_{i,n}')^2)^{\frac{5}{2}}}  
               &- 5 \frac{(u_{i,n}'')^2u_{i,n}'\varphi'}{(1+(u_{i,n}')^2)^{\frac{7}{2}}} \\
               &+ \lambda \frac{u_{i,n}'\varphi'}{(1+(u_{i,n}')^2)^{\frac{1}{2}}} + \frac{w_{i,n}\varphi}{(1+(u_{i-1,n}')^2)^{\frac{1}{2}}} \bigg] \;\mathrm{d}x \ge 0 
\end{split}
\end{align}
for all $\varphi \in C^{\infty}_{{\rm c}}(I)$ with $\varphi \ge 0$. 
We observe from \eqref{Geq:4.14} that $L_{i,n}$ is a nonnegative distribution on $C^{\infty}_{\rm{c}}(I)$. 
Therefore, by the Riesz representation theorem (see e.g.\ \cite[2.14 Theorem]{Rud}), we find a nonnegative Radon measure $\mu_{i,n}$ such that
\begin{equation} 
\label{Geq:4.15}
L_{i,n}(\varphi) = \int_I \varphi \;\mathrm{d}\mu_{i,n}
\end{equation}
for all $\varphi\in C^{\infty}_{{\rm c}}(I)$ and $i=1, \ldots, n$. 
Since $u_{i,n}, \psi \in C(\bar{I})$, the noncoincidence set defined by 
$$
\mathcal{N}_{i,n}:=\{ x \in I \mid u_{i,n}(x) > \psi(x) \}
$$
is open in $I$. 
In addition, the measure $\mu_{i,n}$ has a support on $I \setminus \mathcal{N}_{i,n}$, i.e. 
\begin{align} \label{Geq:4.16}
\mu_{i,n}(\mathcal{N}_{i,n})=0. 
\end{align}
Indeed, for any $\varphi\in C^{\infty}_{{\rm c}}(I)$ with ${\rm supp}\, \varphi \subset \mathcal{N}_{i,n}$, we have $u_{i,n}\pm\varepsilon \varphi\geq \psi$ in $I$ for $\varepsilon > 0$ small enough. 
Thus we have $L_{i,n}(\varphi)=0$, and then \eqref{Geq:4.15} yields 
\begin{equation}
\label{Geq:4.17}
\int_I \varphi \; \mathrm{d}\mu_{i,n} = 0 \quad \text{for all } \ \varphi\in C^{\infty}_{{\rm c}}(I) \ \text{ with } \  \supp \varphi \subset \mathcal{N}_{i,n}. 
\end{equation}
Thus \eqref{Geq:4.16} follows from \eqref{Geq:4.17} (see e.g.\ \cite[Chapter II, Theorem 6.9]{KSbook}). 

The following lemma immediately follows from \eqref{eq:psi_condition} and the fact that $u_{i,n}\in K$.

\begin{lemma} \label{Gtheorem:4.7}
There exists $\delta>0$ such that for any $n\in\mathbf{N}$ and $i = 1,\ldots, n$
$$
(0,\delta)\cup(1-\delta,1) \subset \mathcal{N}_{i,n}. 
$$
\end{lemma}
\begin{proof}
Since $\psi\in C(\bar{I})$ and $\psi(0)<0$, there exists $\delta_0 \in I$ such that 
\begin{equation} 
\label{Geq:4.18}
\psi(x) < \frac{3}{4}\,\psi(0) \quad \text{for}\quad x \in [0,\delta_0). 
\end{equation}
Moreover, by Sobolev's embedding theorem and Lemma \ref{Gtheorem:4.4} we have 
\begin{equation}
\label{Geq:4.19}
 \Vert u_{i,n}\Vert_{W^{1, \infty}(I)} \leq C  \Vert u_{i,n}\Vert_{H^2(I)} \leq C\rho  
\end{equation}
for $n \in \mathbf{N}$ and $i=1, \ldots, n$. 
This together with the fact that $u_{i,n}(0)=0$ yields 
$$
\sup_{x\in(0,\delta_0)}\frac{| u_{i,n}(x)|}{x} \leq \Vert u_{i,n}\Vert_{W^{1, \infty}(I)}\leq C\rho.
$$
Hence, taking $\delta \in(0,\delta_0)$ small enough to satisfy $\delta<-\psi(0) /4C\rho$, we obtain
\begin{equation}
\label{Geq:4.20}
u_{i,n}(x) \geq \frac{1}{4}\,\psi(0) \quad \text{for} \quad x \in [0,\delta].
\end{equation}
Thus, by \eqref{Geq:4.18} and \eqref{Geq:4.20} we see that the above $\delta$ satisfies $(0,\delta) \subset \mathcal{N}_{i,n}$. 
In addition \eqref{Geq:4.18} and \eqref{Geq:4.19} imply that $\delta$ is independent of $i$ and $n$. 
Following the same argument, by replacing $\delta$ with smaller one if needed, we can show that $(1-\delta,1)\subset \mathcal{N}_{i,n}$. 
 \end{proof}

Thanks to Lemma~\ref{Gtheorem:4.7} combined with \eqref{Geq:4.16} it follows that 
\[
\mu_{i,n}(I) = \mu_{i,n}(0,\delta)  + \mu_{i,n}([\delta,1-\delta]) +  \mu_{i,n}(1-\delta,1) = \mu_{i,n}([\delta,1-\delta]).
\]
This allows us to deduce a uniform estimate of $\mu_{i,n}(I)$ as follows. 

\begin{lemma} \label{Gtheorem:4.8} 
There exists a constant $C_1>0$ such that for any $n \in \mathbf{N}$
$$
\tau_{n}\sum_{i=1}^{n} \mu_{i,n}(I)^2 \leq C_1. 
$$
\end{lemma}
\begin{proof}
Let $\delta>0$ be the constant obtained in Lemma~\ref{Gtheorem:4.7}.
Fix $\zeta\in C^{\infty}_{{\rm c}}(I)$ with $\zeta\equiv1$ in $[\delta,1-\delta]$ and $0\leq \zeta \leq 1$ in $\bar{I}$. 
From \eqref{Geq:4.15} it follows that 
\[
\mu_{i,n}(I) =\mu_{i,n}([\delta,1-\delta]) 
 \le \int_I\zeta \;\mathrm{d}\mu_{i,n} 
  = L_{i,n}(\zeta),
\]
and from H\"older's inequality we deduce that 
\begin{align*}
L_{i,n}(\zeta)   
& \le 2 \Vert u_{i,n}''\Vert_{L^2(I)} \Vert \zeta''\Vert_{L^2(I)}  
+5 \Vert u_{i,n}''\Vert_{L^2(I)}^2 \Vert u_{i,n}'\Vert_{L^{\infty}(I)} \Vert \zeta'\Vert _{L^{\infty}(I)} \\
& \quad   +\lambda \textcolor{black}{\Vert u_{i,n}''\Vert_{L^{2}(I)} \Vert \zeta\Vert_{L^{2}(I)}}+  \Vert \zeta \Vert_{L^{2}(I)} \Big\|\tfrac{w_{i,n}}{\sqrt{1+|u_{i-1,n}'|^2}}\Big\|_{L^2(I)}.
\end{align*}
These estimates lead to 
$$
\mu_{i,n}(I) \leq C \bigg( \sup_{1\leq i\leq n}\|u_{i,n}\|_{H^2(I)} + \sup_{1\leq i\leq n}\|u_{i,n}\|_{H^2(I)}^2 + \Big\|\tfrac{w_{i,n}}{\sqrt{1+|u_{i-1,n}'|^2}}\Big\|_{L^2(I)} \bigg),
$$
where $C>0$ is a constant depending only on $\delta,\lambda,M_0$.
Noting that
\[ \tau_n \sum_{i=1}^n \Big\|\tfrac{w_{i,n}}{\sqrt{1+|u_{i-1,n}'|^2}}\Big\|_{L^2(I)}^2 = \int_0^T\!\!\!\int_I\frac{(\partial_t u_n)^2}{\sqrt{1+(\underline{u}_n')^2}}\;\mathrm{d}x\mathrm{d}t,
\]
we deduce from Lemma~\ref{Gtheorem:4.4} that 
$$
\tau_n \sum_{i=1}^n \mu_{i,n}(I)^2 
 \le C\big( T\rho^2 + T\rho^4 + \rho^2 \big) =:C_1, 
$$
which completes the proof.
 \end{proof}

Next we improve the regularity of the approximate solution $\bar{u}_n$.

\begin{lemma} \label{Gtheorem:4.10}
Let $\bar{u}_n$ be the piecewise constant interpolation of $\{ u_{i,n}\}$. 
Then, for any $p\in[2,\infty]$ there exists $C>0$ such that for any $n \in \mathbf{N}$
\begin{align}\label{eq:250205-1}
\Vert \bar{u}'''_n \Vert_{L^{\frac{p}{p-1}}(0,T;L^p(I))} \le C(T+1).
\end{align}
In particular, $\bar{u}_n \in L^{\frac{p}{p-1}}(0,T;W^{3,p}(I))$ for any $p\in[2,\infty]$.
\end{lemma}
\begin{proof}
For the measure $\mu_{i,n}$ obtained in \eqref{Geq:4.15}, we define a function $m_{i,n}$ as 
$$
m_{i,n}(x):=\mu_{i,n}(0,x) 
\quad  \text{for} \quad x \in I.
$$
Then $m_{i,n}$ is of bounded variation on $I$. 
Using integration by parts for Lebesgue-Stieltjes integrals  induced by $m_{i,n}$ (\cite[Chapter III, Theorem 14.1]{saks}), we have 
\begin{equation*}
\int_I \varphi\;\mathrm{d}\mu_{i,n}(x) =- \int_I m_{i,n}(x) \varphi'(x) \;\mathrm{d}x
\end{equation*}
for $\varphi\in C^{\infty}_{\rm c}(I)$. 
This and \eqref{Geq:4.15} imply
\begin{align}\label{Geq:4.23}
\begin{split}
2 \int_I \dfrac{u_{i,n}''}{(1+(u_{i,n}')^2)^{\frac{5}{2}}} \varphi'' \;\mathrm{d}x
=&\, 
5\int_I \frac{(u_{i,n}'')^2u_{i,n}'\varphi'}{(1+(u_{i,n}')^2)^{\frac{7}{2}}} \;\mathrm{d}x - \lambda \int_I \frac{u_{i,n}'\varphi'}{(1+(u_{i,n}')^2)^{\frac{1}{2}}}\;\mathrm{d}x \\
&\quad -\int_I \frac{w_{i,n}\varphi}{(1+(u_{i-1,n}')^2)^{\frac{1}{2}}}  \;\mathrm{d}x - \int_I m_{i,n} \varphi' \;\mathrm{d}x
\end{split}
\end{align}
for $\varphi\in C^{\infty}_{\rm c}(I)$. 
By a density argument \eqref{Geq:4.23} also holds for $\varphi\in H^2_0(I)$. 

Fix $\eta\in C^{\infty}_{\rm c}(I)$ arbitrarily and define $\varphi_1:I\to\mathbf{R}$ by
\begin{align*}
\varphi_1(x) &:= \int_{0}^x\!\!\!\int_{0}^y \eta (s)\;\mathrm{d}s\mathrm{d}y+ \alpha x^2 + \beta x^3, \quad x \in I,
\end{align*}
where 
\begin{align*}
\alpha&:= \int_I \eta (y)\;\mathrm{d}y -3 \int_I\int_{0}^y \eta (s)\;\mathrm{d}s \mathrm{d}y,  \quad 
\beta:= -\alpha - \int_I\int_{0}^y \eta (s)\;\mathrm{d}s\mathrm{d}y.
\end{align*}
Then $\varphi_1 \in H^2_0(I)$ and there exists a universal constant $C$ such that 
$\Vert\varphi_1\Vert_{C^1(\bar{I})} \leq C\Vert\eta\Vert_{L^1(I)}$, $|\alpha|\leq C\Vert\eta\Vert_{L^1(I)}$, and $|\beta| \leq C \Vert\eta\Vert_{L^1(I)}$.
Taking $\varphi_1$ as $\varphi$ in \eqref{Geq:4.23}, 
we have 
\begin{align*} 
\bigg| \int_I \frac{u''_{i,n}}{(1+(u_{i,n}')^2)^{\frac{5}{2}}}  \eta \;\mathrm{d}x \bigg| 
 & \le  C\bigg( 1 +  \bigg\Vert \frac{w_{i,n}}{(1+(u_{i-1,n}')^2)^{\frac{1}{2}}}\bigg\Vert_{L^1(I)} +  \mu_{i,n}(I)  \bigg) \Vert\eta\Vert_{L^1(I)} \\
 &\quad:=C_{i,n}\Vert\eta\Vert_{L^1(I)}, 
\end{align*}
which implies that $u''_{i,n}(1+(u_{i,n}')^2)^{-\frac{5}{2}} \in L^\infty(I)$ and 
\begin{equation} 
\label{Geq:4.26}
\bigg\Vert \frac{u''_{i,n}}{(1+(u_{i,n}')^2)^{\frac{5}{2}}}\bigg\Vert_{L^{\infty}(I)} \leq C_{i,n}.
\end{equation}
Moreover, combining this with the fact that $\Vert u_{i,n}'\Vert_{L^{\infty}(I)}\leq 2M_0$, we obtain
\begin{align}\label{Geq:4.27}
\left\Vert u''_{i,n}\right\Vert_{L^{\infty}(I)} \le (1+4M_0^2)^{\frac{5}{2}} C_{i,n}. 
\end{align}
We turn to the uniform $W^{3,p}$-estimate on $u_{i,n}$. 
Fix $\eta\in C^{\infty}_{\rm c}(I)$ arbitrarily and define $\varphi_2:I\to\mathbf{R}$ by
\[
\varphi_2(x) := \int_0^x \eta (y)\;\mathrm{d}y +\left( \int_{0}^{1} \eta (y)\;\mathrm{d}y \right) \left(-3x^2  +2x^3\right),\quad x\in I.
\]
Then $\varphi_2 \in H^2_0(I)$ and by using a universal constant $C$ we have $\Vert\varphi_2\Vert_{L^\infty(I)} \leq C\Vert\eta\Vert_{L^1(I)}$, and $\Vert\varphi_2'\Vert_{L^p(I)} \leq C\Vert\eta\Vert_{L^p(I)}$ for any $p\in[1,\infty]$.
Taking $\varphi_2$ as $\varphi$ in \eqref{Geq:4.23} and using the $L^\infty$-bound on $u_{i,n}''$ in \eqref{Geq:4.26}, we obtain  
\begin{align}
2 \int_I \dfrac{u_{i,n}''}{(1+(u_{i,n}')^2)^{\frac{5}{2}}} \eta' \;\mathrm{d}x
=&\, -\bigg(\int_I\eta\;\mathrm{d}y \bigg) \bigg(\int_I\dfrac{u_{i,n}''}{(1+(u_{i,n}')^2)^{\frac{5}{2}}}(-6x+6x^2)\;\mathrm{d}x \bigg) \label{eq:250626-1}\\ 
&\quad+5\int_I \frac{(u_{i,n}'')^2u_{i,n}'\varphi_2'}{(1+(u_{i,n}')^2)^{\frac{7}{2}}} \;\mathrm{d}x - \lambda \int_I \frac{u_{i,n}'\varphi_2'}{(1+(u_{i,n}')^2)^{\frac{1}{2}}}\;\mathrm{d}x \notag\\
&\quad -\int_I \frac{w_{i,n}\varphi_2}{(1+(u_{i-1,n}')^2)^{\frac{1}{2}}}  \;\mathrm{d}x - \int_I m_{i,n} \varphi_2' \;\mathrm{d}x.\notag
\end{align}
Now fix an arbitrary $p\in[2,\infty]$. 
Recalling from \eqref{Geq:4.7} that $\mathcal{B}(u_{i,n}) \leq \mathcal{E}_\lambda(u_{i,n}) \leq \mathcal{E}_\lambda(u_0)$, we have 
\begin{align}\label{eq:H2_estimate_step}
\textcolor{black}{
\|u_{i,n}''\|_{L^2(I)}^2 \leq (1+4M_0^2)^{\frac{5}{2}}\int_I \frac{(u''_{i,n})^2}{(1+(u_{i,n}')^2)^{\frac{5}{2}}}\;\mathrm{d}x \leq (1+4M_0^2)^{\frac{5}{2}}\mathcal{E}_\lambda(u_0).
}
\end{align}
This and \eqref{Geq:4.27} yield that
\begin{align*}
\bigg|\int_I \frac{(u_{i,n}'')^2u_{i,n}'\varphi_2'}{(1+(u_{i,n}')^2)^{\frac{7}{2}}} \;\mathrm{d}x \bigg|
&\leq 2M_0\|u_{i,n}''\|_{L^2(I)}^{\frac{2}{p}}\|u_{i,n}''\|_{L^\infty(I)}^{\frac{2p-2}{p}} \|\varphi_2'\|_{L^{\frac{p}{p-1}}(I)} \\ 
&\leq CC_{i,n}^{{\frac{2p-2}{p}}}\|\varphi_2'\|_{L^{\frac{p}{p-1}}(I)}. 
\end{align*}
Therefore, it follows from \eqref{eq:250626-1} that  
\begin{align*}
2&\bigg| \int_I \dfrac{u_{i,n}''}{(1+(u_{i,n}')^2)^{\frac{5}{2}}} \eta' \;\mathrm{d}x\bigg| \\
&\leq \, 
C\|\eta\|_{L^1(I)}  \bigg\| \dfrac{u_{i,n}''}{(1+(u_{i,n}')^2)^{\frac{5}{4}}} \bigg\|_{L^2(I)} 
+5CC_{i,n}^{{\frac{2p-2}{p}}}\|\varphi_2'\|_{L^{\frac{p}{p-1}}(I)} \\
&\ \ + \lambda \|\varphi_2'\|_{L^1(I)} 
+ \bigg\Vert \frac{w_{i,n}}{(1+(u_{i-1,n}')^2)^{\frac{1}{2}}}\bigg\Vert_{L^1(I)}\|\varphi_2\|_{L^\infty(I)}  +  \mu_{i,n}(I) \|\varphi_2'\|_{L^1(I)} \\
&\leq \, C \bigg( 1 + C_{i,n}^{{\frac{2p-2}{p}}}  + \bigg\Vert \frac{w_{i,n}}{(1+(u_{i-1,n}')^2)^{\frac{1}{2}}}\bigg\Vert_{L^1(I)} + \mu_{i,n}(I) \bigg)\|\eta\|_{L^{\frac{p}{p-1}}(I)} \\
&\leq \, C \Big(  C_{i,n}^{{\frac{2p-2}{p}}}  + C_{i,n}\Big)\|\eta\|_{L^{\frac{p}{p-1}}(I)}.
\end{align*}
We may assume $C_{i,n}\geq1$ by replacing $C_{i,n}$ with $C_{i,n}+1$. Since $(2p-2)/p\geq1$, we have 
$$
\bigg| \int_I \frac{u''_{i,n}}{(1+(u_{i,n}')^2)^{\frac{5}{2}}}  \eta' \;\mathrm{d}x \bigg| \le C C_{i,n}^{{\frac{2p-2}{p}}}\|\eta\|_{L^{\frac{p}{p-1}}(I)}, 
$$
so that we find that $f_{i,n}:=u''_{i,n}(1+(u_{i,n}')^2)^{-\frac{5}{2}} \in W^{1,p}(I)$ and  
$\|f_{i,n}'\|_{L^p(I)} \leq CC_{i,n}^{\frac{2p-2}{p}}$.
Hence, by fact that $\|(u_{i,n}'')^2\|_{L^p(I)} \leq \|u_{i,n}''\|_{L^2(I)}^{\frac{2}{p}}\|u_{i,n}''\|_{L^\infty(I)}^{\frac{2p-2}{p}}$ it follows that 
\begin{align} 
\begin{split}
\label{eq:W3p-est_i,n}
\Vert u_{i,n}'''\Vert_{L^p(I)}  &\leq (1+4M_0^2)^{\frac{5}{2}}\|f_{i,n}'\|_{L^p(I)} +  \frac{5}{2}(1+4M_0^2)^{\frac{3}{2}}\|f_{i,n}u_{i,n}''\|_{L^p(I)} 
\leq  CC_{i,n}^{{\frac{2p-2}{p}}}.
\end{split}
\end{align}
This leads to 
\begin{align*}
\int_0^T \Vert\bar{u}_{n}'''\Vert^{\frac{p}{p-1}}_{L^p(I)} \;\mathrm{d}t 
&= \sum_{i=1}^n \int_{(i-1)\tau_n}^{i\tau_n}\Vert u_{i,n}'''\Vert_{L^p(I)}^{\frac{p}{p-1}} \;\mathrm{d}t 
\leq C \sum_{i=1}^n \int_{(i-1)\tau_n}^{i\tau_n} C_{i,n}^2 \;\mathrm{d}t
\\
 &\le C \tau_n \sum_{i=1}^n \Big(1+\mu_{i,n}(I)^2\Big) +C\sum_{i=1}^n\int_{(i-1)\tau_n}^{i\tau_n}\int_I \frac{w_{i,n}^2}{(1+(u'_{i-1,n})^2)^{\frac{1}{2}}} \;\mathrm{d}x\mathrm{d}t \\
&\leq C(T+1), 
\end{align*}
where in the last inequality we used Lemmas~\ref{Gtheorem:4.4} and \ref{Gtheorem:4.8}. 
In addition, \eqref{Geq:4.27} yields
\begin{align}\notag
\|u_{i,n}''\|_{L^p(I)} 
\leq \|u_{i,n}''\|_{L^\infty(I)}^{\frac{p-2}{p}}\|u_{i,n}''\|_{L^2(I)}^{\frac{2}{p}}
\leq CC_{i,n}^{\frac{p-2}{p}}.
\end{align}
By Lemma~\ref{Gtheorem:2.1} and the fact that $u_{i,n}\in \tilde{K}$, we have $\|u_{i,n}\|_{W^{1,\infty}(I)}\leq 4M_0$.
Hence,
\[
\int_0^T \Vert\bar{u}_{n}\Vert^{\frac{p}{p-1}}_{W^{2,p}(I)} \;\mathrm{d}t 
\leq C\sum_{i=1}^n \int_{(i-1)\tau_n}^{i\tau_n}C_{i,n}^2 \;\mathrm{d}t \leq C(T+1).
\]
Thus we have also shown that $\bar{u}_{n} \in L^{\frac{p}{p-1}}(0,T;W^{3,p}(I))$. 
 \end{proof}

\subsection{Convergence} 

\begin{lemma} \label{lem:conv_u_n} 
Let ${u}_{n}$ be the piecewise linear interpolation of $\{ u_{i,n}\}$ defined in Definition~\ref{def:interpolant}.
Then, there exists 
\begin{align}\label{eq:mild_regularity_u}
u \in L^{\infty}(0,T;H^2(I)) \cap H^1(0,T;L^{2}(I)) 
\end{align}
such that $u$ possesses a continuous representation on $[0,T]$ with values in $C^{1,1}(\bar{I})$, i.e.\ $u \in C([0,T]; C^{1,1}(\bar{I}))$, and up to a subsequence, 
\begin{align}
&{u}_{n} \rightharpoonup u \quad \text{weakly$^*$ in} \quad L^{\infty}(0,T;H^2(I)), \label{eq:250126-5}\\
&u_{n} \rightharpoonup u \quad \text{weakly in} \quad H^1(0,T;L^2(I)), \label{eq:250126-2}\\
&u_n  \to  u \quad  \text{in} \quad C([0,T];C^{1,1}(\bar{I})) \label{eq:s-limit_u}.
\end{align}
\end{lemma}
\begin{proof}
First, by \eqref{Geq:4.6} we can deduce the uniform estimate as
\begin{align}\label{eq:250218-1}
\sup_{t\in[0,T]}\|{u}_n(\cdot,t)\|_{H^2(I)} \leq \sup_{0\leq i \leq n}\|u_{i,n}\|_{H^2(I)} \leq \rho.
\end{align}
Thus, we can find $u \in L^{\infty}(0,T;H^2(I))$ such that \eqref{eq:250126-5} holds up to a subsequence.

Next we show \eqref{eq:250126-2}. 
Combining \eqref{Geq:4.6} with the fact that $\sup_{n\in\mathbf{N}}\|\underline{u}_n'\| \leq 2M_0$, we see that $\partial_t u_n$ is uniformly bounded in $L^2(0,T;L^{2}(I))$. 
Therefore, there exists $w_\infty \in L^2(0,T;L^{2}(I))$ such that, up to a subsequence, $\partial_t u_n$ converges to $w_\infty$ weakly in $L^2(0,T;L^{2}(I))$. 
Then, passing to the weak limit in 
\[
 \int_0^T\!\!\int_I  u_n\, \partial_t \phi \;\mathrm{d}x\mathrm{d}t = -\int_0^T\!\!\int_I \partial_t u_n  \phi \;\mathrm{d}x\mathrm{d}t, 
\]
where $\phi \in C^\infty_{\rm c}(I\times(0,T))$ is a test function, we find that 
\[
 \int_0^T\!\!\int_I  u\, \partial_t \phi \;\mathrm{d}x\mathrm{d}t = -\int_0^T\!\!\int_I w_\infty  \phi \;\mathrm{d}x\mathrm{d}t. 
\]
This implies that the limit $u$ obtained in \eqref{eq:250126-5} satisfies $u\in H^1(0,T;L^2(I))$ and $\partial_t u =w_\infty$, so that the limit $u$ satisfies \eqref{eq:mild_regularity_u} and \eqref{eq:250126-2} has been shown.

It remains to show \eqref{eq:s-limit_u}.
Thanks to Aubin--Lions--Simon compactness theorem (see e.g.\ \cite[Theorem II.5.16]{BF_2013}), the embedding of 
\[
\Big\{v\in L^\infty(0,T;H^2(I)), \ \frac{\partial v}{\partial t} \in L^2(0,T;L^2(I))\Big\}
\]
in $C([0,T]; C^{1,1}(\bar{I}))$ is compact. 
Thus, by \eqref{eq:mild_regularity_u} we have $u\in C([0,T];C^{1,1}(\bar{I}))$. 
Then \eqref{eq:s-limit_u} immediately follows.
 \end{proof}

\begin{lemma} \label{lem:conv_tilde-u_n} 
Let $u$  be the limit obtained in Lemma~\ref{lem:conv_u_n}. 
Then, for any $p\in [2,\infty)$ $u \in L^{\frac{p}{p-1}}(0,T;W^{3,p}(I))$.
In addition, there are subsequences of $\bar{u}_{n}$ and $\underline{u}_{n}$ (without relabeling) such that  
\begin{align}
&\bar{u}_{n} \rightharpoonup u \quad \,\,\,\,\,\, \text{weakly$^*$ in} \quad L^{\infty}(0,T;H^2(I)), \label{Geq:4.33}\\
&\bar{u}_{n} \rightharpoonup u\quad \quad \text{weakly in} \quad L^{\frac{p}{p-1}}(0,T;W^{3,p}(I)), \label{eq:250126-1} \\
&\bar{u}_n  \to  u \quad \,\,\,\,\,\,\, \text{in} \quad L^2(0,T;W^{2,\infty}(I)) \label{eq:s-limit_tilde_u} \\
&\underline{u}_n  \to  u \quad \,\,\,\,\,\,\, \text{in} \quad L^2(0,T;C^{1}(\bar{I})). \label{eq:limit_under_u} 
\end{align}
\end{lemma}
\begin{proof}
As in \eqref{eq:250218-1}, we can check that $\bar{u}_n$ is uniformly bounded in $L^\infty(0,T;H^2(I))$.
Thus there exists $u_* \in L^\infty(0,T;H^2(I))$ such that, up to subsequence, $\bar{u}_n$ converges to $u_*$ weakly$^*$ in $L^{\infty}(0,T;H^2(I))$.
In the following we show that $u_*$ coincides with the limit $u$ obtained by Lemma~\ref{lem:conv_u_n}. 
By definition we have $\bar{u}_n(x,t) - u_n(x,t) = (t-i\tau_n)w_{i,n}$ for each $i\in\{1,\ldots,n\}$ and $t\in((i-1)\tau_n,i\tau_n]$, so that 
\begin{align}
\begin{split}\label{eq:250130-1}
\|\bar{u}_n - u_n\|_{L^2(0,T;L^2(I))}^2
&= \sum_{i=1}^n\int_{(i-1)\tau_n}^{i\tau_n}\int_I|\bar{u}_n - u_n|^2\;\mathrm{d}x\mathrm{d}t \\
&= \sum_{i=1}^n\int_{(i-1)\tau_n}^{i\tau_n}\int_I(t-i\tau_n)^2|w_{i,n}|^2\;\mathrm{d}x\mathrm{d}t \\
&\leq \tau_n^2 (1+4M_0^2)^{\frac{1}{2}}\sum_{i=1}^n\int_{(i-1)\tau_n}^{i\tau_n}\int_I\frac{|w_{i,n}|^2}{(1+(u_{i-1,n}')^2)^{\frac{1}{2}}}\;\mathrm{d}x\mathrm{d}t \\
&\leq \rho^2(1+4M_0^2)^{\frac{1}{2}}\tau_n^2, 
\end{split}
\end{align}
where in the last inequality we used Lemma~\ref{Gtheorem:4.4}.
Thus it follows that
\begin{align}\label{eq:250126-4}
\|\bar{u}_n - u_n\|_{L^2(0,T;L^2(I))} \to0, 
\end{align}
This shows that the (strong) limits of $\bar{u}_n$ and $u_n$ coincide, so that $u=u_*$. 

Similarly the claim \eqref{eq:250126-1} follows. 
Indeed, we observe from Lemma~\ref{Gtheorem:4.10} that $\bar{u}_n$ is uniformly bounded in $L^\frac{p}{p-1}(0,T;W^{3,p}(I))$ for an arbitrary $p\in[2,\infty)$, and hence there exists $\bar{u}_* \in L^\frac{p}{p-1}(0,T;W^{3,p}(I))$ such that, up to a subsequence, $\bar{u}_n$ converges to $\bar{u}_*$ weakly in $L^\frac{p}{p-1}(0,T;W^{3,p}(I))$. 
This and \eqref{eq:250126-4} imply that $\bar{u}_* = u$, and hence $u$ lies in $L^\frac{p}{p-1}(0,T;W^{3,p}(I))$ for any  $p\in[2,\infty)$ and \eqref{eq:250126-1} has been shown.

Next we consider \eqref{eq:s-limit_tilde_u}.
Lemma~\ref{Gtheorem:2.1} combined with \eqref{Geq:2.1} yields 
\[
\|\bar{u}_n(\cdot,t)-u(\cdot,t)\|_{W^{1,\infty}(I)}\leq 2\sqrt{2}\|\bar{u}_n(\cdot,t)-u(\cdot,t)\|_{L^2(I)}^{\frac{1}{4}}\|\bar{u}_n(\cdot,t)-u(\cdot,t)\|_{H^2(I)}^{\frac{3}{4}}
\]
for a.e.\ $t\in(0,T)$. 
Since $u\in L^\infty(0,T;H^2(I))$ by Lemma~\ref{lem:conv_u_n}, and $\|\bar{u}_n(\cdot,t)\|_{H^2(I)}\leq\rho$ as in \eqref{eq:250218-1}, we have 
\begin{align}\label{eq:250126-6}
\int_0^T\|\bar{u}_n(\cdot,t)-u(\cdot,t)\|_{W^{1,\infty}(I)}^2\;\mathrm{d}t 
&\leq C\|\bar{u}_n(\cdot,t)-u(\cdot,t)\|_{L^2(0,T;L^2(I))}. 
\end{align}
In addition, by Proposition~\ref{Gtheorem:2.2} we see that 
\begin{align*}
&\|\bar{u}''_n(\cdot,t)-u''(\cdot,t)\|_{L^\infty(I)}\\
\leq\,& C\|\bar{u}_n'''(\cdot,t)-u'''(\cdot,t)\|_{L^2(I)}^{\frac{3}{4}}\|\bar{u}_n(\cdot,t)-u(\cdot,t)\|_{L^2(I)}^{\frac{1}{4}} +C\|\bar{u}_n(\cdot,t)-u(\cdot,t)\|_{L^2(I)}
\end{align*}
for a.e.\ $t\in(0,T)$. 
Hence, 
\begin{align*}
&\int_0^T\|\bar{u}''_n(\cdot,t)-u''(\cdot,t)\|_{L^\infty(I)}^2\;\mathrm{d}t\\
\leq\,& 
C \|\bar{u}_n'''-u'''\|_{L^2(0,T;L^2(I))}^{\frac{3}{2}}\|\bar{u}_n-u\|_{L^2(0,T;L^2(I))}^{\frac{1}{2}}
+C \|\bar{u}_n-u\|_{L^2(0,T;L^2(I))}^2 \\
\leq\,&C \|\bar{u}_n-u\|_{L^2(0,T;L^2(I))}^{\frac{1}{2}} + C \|\bar{u}_n-u\|_{L^2(0,T;L^2(I))}^2, 
\end{align*}
where in the last inequality we used Lemma~\ref{Gtheorem:4.10} and that $u\in L^2(0,T;H^3(I))$. 
Applying \eqref{eq:250126-4} to the above estimate and \eqref{eq:250126-6}, we obtain \eqref{eq:s-limit_tilde_u}.

It remains to show \eqref{eq:limit_under_u}. 
By \eqref{eq:s-limit_tilde_u}, it suffices to prove
\begin{align}\label{eq:250130-3}
    \|\bar{u}_n-\underline{u}_n\|_{L^2(0,T;C^{1}(\bar{I}))} \to 0.
\end{align}
In view of Lemma~\ref{Gtheorem:2.1} we see that
\begin{align*}
    &\|\bar{u}_n-\underline{u}_n\|_{L^2(0,T;C^{1}(\bar{I}))}^2 
    \leq C \int_0^T\|\bar{u}_n''(\cdot,t)-\underline{u}_{n}''(\cdot,t)\|_{L^2(I)}^{\frac{3}{2}}\|\bar{u}_n(\cdot,t)-\underline{u}_{n}(\cdot,t)\|_{L^2(I)}^{\frac{1}{2}}\mathrm{d}t.
\end{align*}
Recall that $\sup_{t\in(0,T)}\|\bar{u}_n''(\cdot,t)\|_{L^2(I)}<\infty$.
Similarly $\underline{u}_n$ is uniformly bounded in $L^\infty(0,T;H^2(I))$. 
Indeed, by \eqref{eq:H2_estimate_step} we have $\sup_{t\in[0,\tau_n]}\|\underline{u}_{n}''(\cdot,t)\|_{L^2(I)} =\|{u}_{0}''\|_{L^2(I)}$.
In addition, Lemma~\ref{Gtheorem:4.4} yields the uniform estimate of $\sup_{t\in[\tau_n, T]}\|\underline{u}_{n}''(\cdot,t)\|_{L^2(I)}$.
Hence 
\begin{align}\label{eq:250622-1}
    &\|\bar{u}_n-\underline{u}_n\|_{L^2(0,T;C^{1}(\bar{I}))}^2 
    \leq C \int_0^T\|\bar{u}_n(\cdot,t)-\underline{u}_{n}(\cdot,t)\|_{L^2(I)}^{\frac{1}{2}}\mathrm{d}t.
\end{align}
Since $\bar{u}_n(x,t)-\underline{u}_n(x,t)=w_{i,n}\tau_n$ by definition, similar to \eqref{eq:250130-1} we obtain
\begin{align*}
\int_0^T \!\!\int_I|\bar{u}_n-\underline{u}_n|^2\;\mathrm{d}x\mathrm{d}t &\leq \sum_{i=1}^n\int_{(i-1)\tau_n}^{i\tau_n}\int_I \tau_n^2 |w_{i,n}|^2\;\mathrm{d}x\mathrm{d}t 
\leq \rho^2(1+4M_0^2)^{\frac{1}{2}}\tau_n^2, 
\end{align*}
which implies $\|\bar{u}_n-\underline{u}_n\|_{L^2(0,T;L^{2}(I))} \to 0$.
Together with \eqref{eq:250622-1} this gives \eqref{eq:250130-3}.
 \end{proof}


Now we are ready to show Theorem~\ref{thm:existence}.

\begin{proof}[Proof of Theorem~\ref{thm:existence}]
Let $u$ be the limit obtained by Lemma~\ref{lem:conv_u_n}. In what follows we show that this limit $u$ is a weak solution of \eqref{eq:P}. 
Note that the additional regularity follows from Lemmas~\ref{lem:conv_u_n}  and \ref{lem:conv_tilde-u_n}.

First we show that $u\in K_T$.
It follows by Lemma~\ref{lem:conv_tilde-u_n} that $u \in L^{\infty}(0,T;H(I)) \cap H^1(0,T;L^2(I))$. 
In addition, since the piecewise linear interpolation $u_n$ satisfies $u_n \ge \psi$ in ${I} \times (0,T)$, we deduce from \eqref{eq:s-limit_u} that $u \ge \psi$ in ${I}\times (0,T)$. 
Thus we have $u\in K_T$. 

Next we show that $u$ satisfies \eqref{eq:VI}.
Fix $v \in K_T$ arbitrarily. 
Then $v(\cdot,t)\in K$ for a.e.\ $t\in(0,T)$.
By convexity of $K$ it follows that for any $n\in\mathbf{N}$, $i = 1,\ldots, n$ and $\varepsilon \in [0,1)$,  
$$
u_{i,n}(x)+\varepsilon(v (x,t)-u_{i,n}(x) ) \geq \psi(x)  
$$
for any $x \in I$ and a.e.\ $t\in(0,T)$.
Note that $\|v\|_{C^0([0,T];C^1(\bar{I}))}<\infty$ holds since $v\in K_T$.
Then, Lemma \ref{Gtheorem:4.5} implies that for $\varepsilon > 0$ small enough
$$
\Vert u'_{i,n}+\varepsilon(v'(\cdot,t)-u'_{i,n}) \Vert_{L^{\infty}(I)}
\le (1-\varepsilon)\frac{3}{2}M_0 + \varepsilon\Vert v'(\cdot,t)\Vert_{L^{\infty}(I)}
\le 2M_{0}.
$$
Thus $u_{i,n}+\varepsilon(v(\cdot,t)-u_{i,n}) \in \tilde{K}$.
Since $u_{i,n}$ is a minimizer of $G_{i,n}$ in $\tilde{K}$, we have
\begin{equation} 
\label{Geq:5.5}
\frac{d}{d\varepsilon} G_{i,n}\big(u_{i,n}+\varepsilon (v(\cdot,t) - u_{i,n} ) \big) \Big|_{\varepsilon=0} \ge 0 
\end{equation}
for a.e.\ $t\in(0,T)$ and any $i=1, \ldots, n$. 
Integrating \eqref{Geq:5.5} with respect to $t$ on $((i-1)\tau_{n},i\tau_{n})$ and summing over $i=1, \ldots, n$, we have
\begin{equation} 
\label{Geq:5.6}
\int_0^T\!\!\!\int_I\bigg[ 2\frac{\bar{u}_n''(v-\bar{u}_n)''}{(1+(\bar{u}'_n)^2)^{\frac{5}{2}}} 
-5 \frac{|\bar{u}_n''|^2(v-\bar{u}_n)'}{(1+(\bar{u}'_n)^2)^{\frac{7}{2}}} - \lambda\frac{\bar{u}_n'(v-\bar{u}_n)'}{(1+(\bar{u}'_n)^2)^{\frac{1}{2}}} + \frac{\partial_t u_n(v-\bar{u}_n)}{(1+(\underline{u}'_n)^2)^{\frac{1}{2}}} \bigg] \; \mathrm{d}x \mathrm{d}t \ge 0
\end{equation}
for all $v\in K_T$.
In view of \eqref{eq:s-limit_tilde_u}, we see that
\begin{align*}
 2\int_0^T\!\!\!\int_I \frac{\bar{u}_n''(v-\bar{u}_n)''}{(1+(\bar{u}'_n)^2)^{\frac{5}{2}}} \; \mathrm{d}x\mathrm{d}t 
&\to 2\int_0^T\!\!\!\int_I \frac{u''(v-u)''}{(1+({u}')^2)^{\frac{5}{2}}} \; \mathrm{d}x\mathrm{d}t, \\
 5\int_0^T\!\!\!\int_I  \frac{|\bar{u}_n''|^2(v-\bar{u}_n)'}{(1+(\bar{u}'_n)^2)^{\frac{7}{2}}} \;  \mathrm{d}x\mathrm{d}t 
&\to 5\int_0^T\!\!\!\int_I  \frac{|u''|^2(v-u)'}{(1+(u')^2)^{\frac{7}{2}}} \; \mathrm{d}x\mathrm{d}t, \\ 
 \lambda \int_0^T\!\!\!\int_I \frac{\bar{u}_n'(v-\bar{u}_n)'}{(1+(\bar{u}'_n)^2)^{\frac{1}{2}}} \; \mathrm{d}x\mathrm{d}t 
&\to\lambda \int_0^T\!\!\!\int_I \frac{u'(v-u)'}{(1+(u')^2)^{\frac{1}{2}}}\, \mathrm{d}x\mathrm{d}t, 
\end{align*}
as $n \to \infty$ up to a subsequence. 
In addition, it follows from \eqref{eq:250126-2}, \eqref{eq:s-limit_tilde_u}, and \eqref{eq:limit_under_u} that 
\begin{equation*}
\int_0^T\!\!\!\int_I  \frac{\partial_t u_n(v-\bar{u}_n)}{(1+(\underline{u}'_n)^2)^{\frac{1}{2}}}  \; \mathrm{d}x \mathrm{d}t  \to \int_0^T\!\!\!\int_I \frac{\partial_tu (v-u)}{(1+(u')^2)^{\frac{1}{2}}}  \; \mathrm{d}x \mathrm{d}t   \quad \text{as} \quad n\to\infty, 
\end{equation*}
up to a subsequence. 
Thus, extracting a subsequence and letting $n\to\infty$ in \eqref{Geq:5.6}, we observe that $u$ satisfies \eqref{eq:VI}.

The remaining condition (iii) in Definition~\ref{def:weak_sol} immediately follows from the convergence \eqref{eq:s-limit_u} with the fact that $u_n$ satisfies $u_n(\cdot,0)=u_0$. 
 \end{proof}

\appendix 
\section{Elliptic integrals and functions}\label{sect:appendix_Jacobi}
In this appendix we collect the definitions of the elliptic integrals and the Jacobian elliptic integrals, and also collect their properties that we used within this note.
In this note we have used the \emph{incomplete elliptic integrals} 
\[
\mathrm{F}(x,q):=\int_0^x\frac{1}{\sqrt{1-q^2\sin^2\theta}}\;\mathrm{d}\theta \quad \text{and} \quad 
\mathrm{E}(x,q):=\int_0^x{\sqrt{1-q^2\sin^2\theta}}\;\mathrm{d}\theta, 
\]
for $x\in\mathbf{R}$ and $q\in[0,1)$. 
We define the \emph{incomplete elliptic integrals} $\mathrm{K}(q)$ and $\mathrm{E}(q)$ by $\mathrm{K}(q):=\mathrm{F}(\frac{\pi}{2},q)$ and $\mathrm{E}(\frac{\pi}{2},q)$, respectively. 
For $q=1$, $\mathrm{F}(x,1)$ is defined on $(-\frac{\pi}{2}, \frac{\pi}{2})$ and we regard $\mathrm{K}(1)=\infty$ by using the fact $\lim_{x\uparrow\frac{\pi}{2}}\mathrm{F}(x,1)=\infty$.

Next we introduce the Jacobian elliptic functions. 
Let us define the \emph{amplitude function} $\am(x,q)$ by the inverse function of $\mathrm{F}(x,q)$, so that 
\[
x=\int_0^{\am(x,q)}\frac{1}{\sqrt{1-q^2\sin^2\theta}}\;\mathrm{d}\theta. 
\]
For $q\in[0,1)$, the \emph{elliptic sine function} $\sn(x,q)$ and the \emph{elliptic cosine function} $\cn(x,q)$ are defined by
\[\sn(x,q):=\sin{\am(x,q)}, \quad \cn(x,q):=\cos{\am(x,q)}, \quad x\in\mathbf{R}.\]
Also the \emph{delta amplitude function} is defined by
\[
\dn(x,q)=\sqrt{1-q^2\sn^2{(x,q)}}, \quad x\in\mathbf{R}.
\]
We collect some fundamental properties of $\sn$, $\cn$, $\dn$ in the following proposition. 
\begin{proposition} \label{prop:Jacobi}
Let $q\in(0,1)$. 
\begin{itemize}
    \item[(i)] \ $\sn(x,q)$ is an odd $2\mathrm{K}(q)$-antiperiodic function on $\mathbf{R}$ and, in $[-\mathrm{K}(q),\mathrm{K}(q)]$, strictly increasing from $-1$ to $1$. 
    \item[(ii)] \ \ $\cn(x,q)$ is an even $2\mathrm{K}(q)$-antiperiodic function on $\mathbf{R}$ and, in $[0,2\mathrm{K}(q)]$, strictly decreasing from $1$ to $-1$. 
    \item[(iii)] \ \  $\dn(x,q)$ is a $\mathrm{K}(q)$-periodic function on $\mathbf{R}$ and, in $[0,\mathrm{K}(q)]$, strictly decreasing from $1$ to $\sqrt{1-q^2}$. 
\end{itemize}
\end{proposition}

Here let us briefly review a relation between critical points of $\mathcal{E}_\lambda$ and the Jacobian elliptic functions.
If a planar curve $\gamma$ is a critical point, then calculating the first variation, we find that the
signed curvature $k$ of $\gamma$ satisfies $2\partial_s^2k +k^3 -\lambda k=0$ (see e.g.\ \cite[Lemma A.2]{DP14} for the calculation of the first variation). 
This fact together with \cite[Proposition 3.3]{Lin96} implies that the signed curvature of $k$ of every critical point of $\mathcal{E}_\lambda$ is given by either of the following forms:
\begin{itemize}
    \item[(i)] $k(s)=\pm A\cn(\alpha s+s_0, q)$ for some $A\geq0$, $\alpha>0$, $s_0\in\mathbf{R}$, $q\in[0,1]$ such that $A^2=4\alpha^2q^2$ and $\lambda=2\alpha^2(2q^2-1)$, 
    \item[(ii)] $k(s)=\pm A\dn(\alpha s+s_0, q)$ for some $A>0$, $\alpha>0$, $s_0\in\mathbf{R}$, $q\in[0,1]$ such that $A^2=4\alpha^2$ and $\lambda=2\alpha^2(2-q^2)$.
\end{itemize}
(In particular, if $q=1$, then in the cases (i) and (ii) the signed curvature is given by $\cn(\cdot,1)=\dn(\cdot,1)=\sech$).
This fact implies that if $k$ satisfies $2\partial_s^2 k+k^3-\lambda k=0$ with $\lambda=0$, then $k$ is $k\equiv0$, or otherwise there exist $\alpha>0$ and $s_0\in\mathbf{R}$ such that
\[
k(s) = \pm \sqrt{2}\alpha \cn(\alpha s+s_0, \tfrac{1}{\sqrt{2}}).
\]
Here, $\alpha>0$ plays a role as a scaling factor. 
In addition, by the periodicity of $\cn$, we may replace $s_0\in\mathbf{R}$ with $s_0\in [-\mathrm{K}(\tfrac{1}{\sqrt{2}}), \mathrm{K}(\tfrac{1}{\sqrt{2}}))$.
This yields \eqref{eq:rect_curvature}.

\section*{acknowledgement}
The author would like to thank Nobuhito Miyake for helpful comments and for drawing the author's attention to reference \cite{LO20}. 
The author would also like to thank the anonymous referee for helpful comments.
The author is supported by JSPS KAKENHI Grant Numbers JP24K16951.

\bibliographystyle{abbrv}
\bibliography{ref_Yoshizawa}

\end{document}